\documentclass[12pt,a4paper]{article}

\usepackage[T1]{fontenc}
\usepackage[utf8]{inputenc}
\usepackage{lmodern}

\usepackage{amsmath,amssymb,mathtools}
\usepackage{bm}

\usepackage{graphicx}
\usepackage{subcaption}
\usepackage{booktabs}
\usepackage{siunitx}
\usepackage{float}

\usepackage{enumitem}

\usepackage{amsthm}

\theoremstyle{plain}
\newtheorem{theorem}{Theorem}[section]
\newtheorem{lemma}[theorem]{Lemma}
\newtheorem{proposition}[theorem]{Proposition}
\newtheorem{corollary}[theorem]{Corollary}

\theoremstyle{definition}
\newtheorem{definition}[theorem]{Definition}
\newtheorem{assumption}[theorem]{Assumption}

\theoremstyle{remark}
\newtheorem{remark}[theorem]{Remark}

\usepackage[colorlinks=true,
            linkcolor=blue,
            citecolor=blue,
            urlcolor=blue]{hyperref}

\usepackage[top=2.5cm,bottom=2.5cm,left=2.5cm,right=2.5cm]{geometry}

\usepackage{setspace}
\usepackage{fancyhdr}
\begin{document}

\begin{titlepage}
\begin{center}
\vspace*{1.2cm}

{\small\texttt{math.DS} --- Dynamical Systems
\quad|\quad
\textit{cross-list:} \texttt{math.PR}, \texttt{stat.TH}, \texttt{nlin.AO}}\\[0.9cm]

{\LARGE\bfseries Intermittency induced by long memory under stochastic regime switching\par}

\vspace{1.2cm}

{\large Mauricio Herrera-Mar\'in}\\[0.3cm]
{\normalsize Universidad del Desarrollo, Faculty of Engineering, Chile}\\[0.2cm]
{\normalsize \href{mailto:mherrera@udd.cl}{mherrera@udd.cl}
\quad
\href{https://orcid.org/0000-0002-9604-3077}{\texttt{ORCID: 0000-0002-9604-3077}}}

\vspace{0.8cm}
{\normalsize \today}

\vspace{1.0cm}

\begin{abstract}
We study a fundamental instability mechanism in nonlinear, nonlocal dynamical systems arising from the interaction of long-range memory and stochastic regime switching. The dynamics are governed by network-coupled, operator-valued Volterra evolutions with completely monotone memory kernels whose excitation operators and kernel parameters are modulated by an ergodic finite-state continuous-time Markov chain. We formalize a sharp separation between annealed stability (in expectation) and quenched behaviour (along typical sample paths). On the annealed side, we identify an averaged memory gain that yields uniform moment bounds and a memory-adapted Lyapunov functional implying mean-square control under an averaged subcriticality condition. On the quenched side, we show that rare but persistent excursions into supercritical regimes are amplified by memory, producing intermittent macroscopic bursts with heavy-tailed statistics and a deterministic almost sure growth exponent obtained via a subadditive ergodic argument. This establishes an annealed--quenched dichotomy specific to non-Markovian switching systems, where stability in expectation can coexist with pathwise growth and metastable burst phases. We further derive a micro--macro correspondence by proving that a population of regime-modulated self-exciting point processes converges, both annealed and quenched, to the random-coefficient Volterra limit, transferring the burst mechanism from microscopic branching dynamics to macroscopic long-memory flows. Numerical experiments illustrate how burst localization depends on graph geometry and on noncommuting excitation operators.
\end{abstract}

\vspace{0.5cm}


\vspace{0.3cm}


\end{center}
\end{titlepage}

\tableofcontents
\newpage

	\section{Introduction}
	\label{sec:intro}

	Intermittency---the irregular alternation between long quiescent phases and sudden bursts of activity---is a canonical signature of nonlinear dynamics far from equilibrium, most prominently in fully developed turbulence and cascade phenomena where extreme events dominate high-order statistics \cite{frisch1995turbulence,she1994universal}. Since its early dynamical-systems formulations (e.g.\ intermittent routes to turbulence) \cite{pomeau1980intermittent} and its archetypal ``on--off'' bursting mechanisms \cite{platt1993onoff}, intermittency has also become a unifying phenomenology across complex networks and stochastic systems, from burst-like neuronal activity and avalanche statistics \cite{beggs2003neuronal} to volatility clustering and heavy-tailed fluctuations in finance \cite{cont2001empirical}. A central theme behind these manifestations is that rare amplifying episodes can dominate typical trajectories and generate strongly non-Gaussian statistics, even when averaged summaries appear benign.
	
	In this work we show that such bursty, heavy-tailed behaviour can arise \emph{structurally} from the interplay of (i) long-range (fractional/Volterra) memory, which accumulates and reinjects past amplification, and (ii) stochastic regime switching, which randomizes the durations of exposure to supercritical dynamics.
		
	Many nonlinear systems evolve under two ubiquitous features: \emph{nonlocal memory} and \emph{randomly changing environments}.
	Memory arises through delayed feedback, persistent correlations, and nonlocal constitutive laws, and is naturally encoded by
	Volterra or fractional kernels \cite{pruss1993evolution,bazhlekova2000fractional}.
	Environmental changes arise as abrupt shifts in operating conditions, latent states, or external forcing and are often modelled by
	regime-switching processes, typically finite-state Markov chains \cite{mao1999,yin2010}.
	A basic question, central to nonlinear dynamics, is how \emph{stability, amplification, and intermittency} emerge when these mechanisms
	interact.
	
	This paper identifies a mechanism of \emph{memory-driven intermittency} in randomly switching nonlocal dynamics.
	The core observation is simple but powerful: stochastic switching generates a distribution of residence times in each regime, and long
	memory accumulates the effect of rare but persistent visits to amplifying (supercritical) regimes.
	Consequently, stability assessed via averaged metrics can mask large burst-like excursions along typical realisations.
	In particular, we show that an evolution may be stable \emph{in expectation} (annealed sense) while exhibiting \emph{pathwise growth}
	(quenched sense), with heavy-tailed burst statistics generated by long residence times.
	This annealed--quenched separation is intrinsic to nonlocal random evolution and has no direct analogue in either
	(i) switching systems without memory, where excursions do not persist once the regime changes, or
	(ii) deterministic nonlocal systems without stochastic residence-time variability.
	
	Nonlocal feedback models appear across applications, including epidemic spreading with history-dependent contacts
	\cite{keeling2008modeling,pastorsatorras2015epidemic}, neural systems with spike-history effects and synaptic plasticity
	\cite{gerstner2014neuronal,truccolo2005point}, financial activity with self-excitation and volatility clustering
	\cite{engle1982autoregressive,bacry2015hawkes}, and climate variability with persistence and regime shifts
	\cite{franzke2015stochastic,lovejoy2013weather}.
	A canonical microscopic amplification mechanism is provided by \emph{self-exciting point processes} (Hawkes processes), where events
	increase future event rates through a convolution kernel \cite{hawkes1971spectra,bremaud1996}.
	When kernels are heavy-tailed or fractional-type, Hawkes dynamics exhibit long-range dependence and admit rough/Volterra limits
	\cite{jaisson2015limit,jaisson2016rough,chevallier2019}.
	At the macroscopic level this leads naturally to Volterra evolutions governed by resolvent families rather than semigroups
	\cite{pruss1993evolution,bazhlekova2000fractional}.
	A further layer of structure arises in coupled systems (e.g.\ networks), where geometry can route amplification through preferred
	spectral directions, potentially producing localisation and mode-mixing effects \cite{trefethen2005,benaych2018,bordenave2020}.
	
	In deterministic or piecewise-deterministic switching, stability is often inferred from frozen-regime criteria, common Lyapunov
	functions, or suitable averaged generators.
	Long memory breaks this intuition: even if an amplifying regime is visited rarely, its influence persists long after the system has
	left it, and its impact depends crucially on the \emph{duration} of the visit.
	Thus switching rates become part of the effective stability mechanism, not merely a mixing detail.
	This leads to a genuinely non-Markovian random evolution problem in which the regime enters through the memory term, so the joint
	process is not Markov in finite dimension without a (typically infinite-dimensional) augmentation of the state, for instance via the
	Bernstein representation of completely monotone kernels \cite{pruss1993evolution,bazhlekova2000fractional}.
	
	We study operator-valued Volterra evolutions with completely monotone memory kernels, whose excitation operators and kernel
	parameters are modulated by an ergodic finite-state continuous-time Markov chain.
	This framework is sufficiently general to encompass a broad class of nonlocal linearizations and nonlocal mean-field limits, while
	still allowing sharp stability statements.
	A central theme is the distinction between:
	\begin{itemize}
		\item \emph{Annealed stability}: stability in expectation or moments, governed by averaged memory gains and Lyapunov drift estimates.
		\item \emph{Quenched behaviour}: pathwise growth and burst statistics along typical regime realisations, governed by rare residence-time
		events amplified by memory.
	\end{itemize}
	While the annealed/quenched dichotomy is familiar in parts of random dynamical systems, here it arises in a specifically \emph{nonlocal}
	setting where standard semigroup and Markov tools do not directly apply and where the memory term couples the present to a random
	history of regimes.
	
	The paper follows a mechanism-first perspective and makes four contributions.
	\begin{enumerate}[label=(\roman*),leftmargin=0.8cm]
		\item \textit{Pathwise well-posedness for randomly switching Volterra dynamics.}
		We formulate a mild-solution theory along each sample path of the regime process, proving existence, uniqueness, and continuous
		dependence for the resulting random nonlocal evolution.
		\item \textit{Annealed stability via an averaged memory gain.}
		We identify an averaged subcriticality condition expressed through an \emph{averaged fractional branching ratio} (an averaged memory
		gain) and derive finite-horizon moment bounds; under additional dissipativity we construct a memory-adapted Lyapunov functional
		yielding mean-square dissipation estimates.
		\item \textit{Quenched intermittency and heavy-tailed bursts.}
		We show that annealed stability does not control typical sample paths: slow exits from supercritical regimes generate burst
		multipliers with power-law tails.
		Using a subadditive ergodic argument we establish the existence of a deterministic almost sure growth exponent and identify explicit
		parameter windows where annealed stability coexists with pathwise instability, formalising an annealed--quenched dichotomy for
		nonlocal random dynamics.
		\item \textit{Micro--macro correspondence in a random environment.}
		We show that regime-modulated Hawkes processes with long-memory kernels converge to the random-coefficient Volterra dynamics, both
		annealed and quenched, transferring the burst mechanism from microscopic branching fluctuations to macroscopic long-memory flows
		\cite{hawkes1971spectra,bremaud1996,jaisson2015limit,jaisson2016rough,chevallier2019}.
	\end{enumerate}
	
	Fractional/Volterra evolution equations and resolvent families provide the natural functional-analytic framework for memory systems
	\cite{pruss1993evolution,bazhlekova2000fractional}.
	Stochastic switching systems are classically treated via generator averaging, Lyapunov methods, and ergodic arguments
	\cite{mao1999,yin2010}.
	Hawkes processes and their long-memory limits supply a probabilistic foundation for self-excitation and branching-based amplification
	\cite{hawkes1971spectra,bremaud1996,jaisson2015limit,jaisson2016rough,chevallier2019}.
	Our contribution is to isolate a mechanism that requires \emph{both} components---memory accumulation and stochastic residence-time
	variability---and to provide a unified treatment combining resolvent methods, memory Lyapunov functionals, and subadditive ergodic
	tools.
	In coupled settings, we also show numerically that burst localisation depends on geometry and persists beyond commuting/diagonalisable
	structures, where non-normal amplification and mode mixing become relevant \cite{trefethen2005,benaych2018,bordenave2020}.
	
\section{Main results}
\label{sec:main}

We summarise the main theoretical contributions and indicate how they connect.
Precise assumptions, definitions, and proofs are developed in
Sections~\ref{sec:model}--\ref{sec:micro-macro}, with the pathwise well-posedness
result proved in Section~\ref{sec:wp}.

Throughout, $Z(t)$ denotes an ergodic continuous-time Markov chain on a finite
state space $\mathcal Z=\{1,\dots,m\}$ with generator $Q$ and invariant
distribution $\pi$.
We study the randomly switching Volterra evolution on $\mathcal X=H\times H$,
\begin{equation}
	\label{eq:main-switching-main}
	U'(t)
	=
	\mathcal{B}U(t)
	+
	\int_0^t \mathcal{G}_{Z(t)}(t-s)\,U(s)\,ds
	+
	F_{Z(t)}, \qquad t>0,\quad U(0)=U_0\in\mathcal{X},
\end{equation}
where $\mathcal B$ is dissipative, and the memory kernel has the structured form
\begin{equation}
	\label{eq:Gz-structure-main}
	\mathcal{G}_z(t)(x,\lambda)
	=
	\begin{pmatrix}
		0\\ A_z(x+\lambda)
	\end{pmatrix} g_z(t),
	\qquad t>0.
\end{equation}
The regime enters through the Volterra term, so $(U(t),Z(t))$ is typically non-Markovian
in finite dimension (Remark~\ref{rem:nonmarkov}).

Define the regime-dependent memory gain (fractional branching ratio)
\begin{equation}
	\label{eq:rho-def-main}
	\rho_z := \|A_z\|\,G_z,
	\qquad
	G_z:=\int_0^\infty g_z(t)\,dt,
\end{equation}
and its stationary average
\begin{equation}
	\label{eq:rho-bar-main}
	\bar\rho := \sum_{z\in\mathcal Z}\pi_z\,\rho_z .
\end{equation}

\subsection*{Pathwise well-posedness}

\begin{theorem}[Pathwise existence and uniqueness]
	\label{thm:main-wp}
	Under Assumptions~\ref{ass:kernels}--\ref{ass:dissip}, for every $T>0$ and for
	every sample path of $Z(\cdot)$ there exists a unique mild solution
	$U\in C([0,T];\mathcal X)$ of \eqref{eq:main-switching-main}. Moreover, the
	solution depends continuously on the initial condition $U_0$.
\end{theorem}

\subsection*{Annealed stability via averaged memory gain}

\begin{theorem}[Annealed control and memory-Lyapunov dissipation]
	\label{thm:main-annealed}
	Assume that $Z(t)$ is ergodic with invariant distribution $\pi$ and define
	$\bar\rho$ by \eqref{eq:rho-bar-main}.
	\begin{enumerate}[label=(\roman*),leftmargin=0.8cm]
		\item (\emph{Finite-horizon moment control.})
		For every $T>0$ there exists $C_T<\infty$ such that
		\[
		\sup_{0\le t\le T}\mathbb E\|U(t)\|_{\mathcal X}
		\le C_T\big(\|U_0\|+\sup_z\|F_z\|_{\mathcal X}\big).
		\]
		\item (\emph{Mean-square dissipation under an averaged margin.})
		If $F\equiv 0$, $\mathcal B$ is strictly dissipative with constant $\beta>0$
		(Assumption~\ref{ass:dissip-strong}), and $\bar\rho<\beta$, then there exists a
		memory-adapted Lyapunov functional $\mathcal V$ such that
		\[
		\frac{d}{dt}\,\mathbb E\big[\mathcal V_{Z(t)}(t)\big]
		\le
		-\,c\,\mathbb E\|U(t)\|_{\mathcal X}^2
		\qquad (t\ge 0)
		\]
		for some $c>0$, yielding uniform mean-square bounds and integrated dissipation.
	\end{enumerate}
\end{theorem}

\subsection*{Quenched growth and annealed--quenched dichotomy}

\begin{theorem}[Deterministic quenched growth exponent]
	\label{thm:main-quenched}
	Assume that $Z(t)$ is stationary and ergodic and that $F\equiv 0$.
	Then there exists a deterministic constant $\gamma\in[-\infty,\infty)$ such that
	\[
	\limsup_{t\to\infty}\frac{1}{t}\log\|U(t,\omega)\|_{\mathcal X}
	=
	\gamma
	\qquad\text{for $\mathbb P$-almost every }\omega.
	\]
	Moreover, there exist parameter regimes for which annealed control holds while
	$\gamma>0$, i.e.\ annealed stability coexists with quenched instability.
\end{theorem}

\subsection*{Micro--macro correspondence in a switching environment}

\begin{theorem}[Micro--macro limit for regime-driven Hawkes processes]
	\label{thm:main-micro-macro}
	Consider a large population of Hawkes processes whose reproduction parameters
	and kernels are modulated by the same regime process $Z(t)$.
	Then the empirical mean intensity converges to a random-coefficient Volterra
	equation, both in probability (annealed) and almost surely conditional on typical
	environment paths (quenched).
\end{theorem}

The remainder of the paper is organized as follows.
Section~\ref{sec:model} introduces the model and standing assumptions.
Section~\ref{sec:wp} proves pathwise existence and uniqueness.
Section~\ref{sec:mean-stab} develops annealed bounds and memory-Lyapunov dissipation.
Section~\ref{sec:as-stab} studies quenched growth, burst tails, and the
annealed--quenched dichotomy.
Section~\ref{sec:micro-macro} establishes micro--macro limits for regime-driven Hawkes
processes.
Section~\ref{sec:numerics} presents numerical experiments.

\section{Model and standing assumptions}
\label{sec:model}

We introduce the randomly switching nonlocal dynamics studied in this paper.
Long-range memory is encoded by Volterra kernels, while stochastic regime switching is
encoded by a finite-state continuous-time Markov chain. The combination yields a
non-Markovian evolution whose quenched (typical-path) behaviour may differ sharply from
its annealed (averaged) behaviour.

\subsection{State space and regime process}
\label{subsec:state-regime}

Let $H=\mathbb{R}^n$ with inner product $\langle\cdot,\cdot\rangle$ and norm $\|\cdot\|$.
Define the augmented state space
\[
\mathcal X := H \times H, \qquad
\|(x,\lambda)\|_{\mathcal X}^2 := \|x\|^2 + \|\lambda\|^2.
\]
Let $\mathcal Z=\{1,\dots,m\}$ and let $\{Z(t)\}_{t\ge0}$ be an ergodic
continuous-time Markov chain on $\mathcal Z$ with generator
$Q=(q_{zz'})_{z,z'\in\mathcal Z}$ and invariant distribution $\pi=(\pi_z)_{z\in\mathcal Z}$:
\[
\pi^\top Q = 0,\qquad \sum_{z\in\mathcal Z}\pi_z=1,\qquad \pi_z>0.
\]

\subsection{Randomly switching Volterra evolution}
\label{subsec:random-volterra}

We study the random Volterra evolution
\begin{equation}
	\label{eq:main-switching}
	U'(t)
	=
	\mathcal B\,U(t)
	+
	\int_{0}^{t}\mathcal G_{Z(t)}(t-s)\,U(s)\,ds
	+
	F_{Z(t)},
	\qquad t>0,\qquad U(0)=U_0\in\mathcal X,
\end{equation}
where $U(t)=(x(t),\lambda(t))\in\mathcal X$,
$\mathcal B:\mathcal X\to\mathcal X$ is a (local) linear operator,
$F_z\in\mathcal X$ is a regime-dependent forcing, and $\mathcal G_z(\cdot)$ is an
operator-valued memory kernel encoding nonlocal feedback.

We adopt the structured kernel form
\begin{equation}
	\label{eq:Gz-structure}
	\mathcal G_{z}(t)(x,\lambda)
	=
	\begin{pmatrix}
		0\\ A_z(x+\lambda)
	\end{pmatrix} g_z(t),
	\qquad t>0,
\end{equation}
where $A_z\in\mathcal L(H)$ is a bounded excitation operator and
$g_z:(0,\infty)\to[0,\infty)$ is a scalar memory kernel.

\begin{remark}[Non-Markovian coupling]
	\label{rem:nonmarkov}
	Although $Z(t)$ is Markovian, $U(t)$ depends on the full history
	$\{U(s):0\le s\le t\}$ through the Volterra convolution in \eqref{eq:main-switching}.
	Hence $(U(t),Z(t))$ is generally not Markov in finite dimension unless the state is
	augmented by (typically infinite-dimensional) memory variables; see
	Section~\ref{subsec:bernstein}.
\end{remark}

\begin{remark}[Noncommutativity and mode mixing]
	\label{rem:noncomm}
	No commutativity is assumed among $\{A_z\}_{z\in\mathcal Z}$, nor between $A_z$ and
	operators defining $\mathcal B$. This is essential for robustness of burst localisation
	and for capturing regime-induced mode mixing in networked settings.
\end{remark}

\subsection{Standing assumptions}
\label{subsec:assumptions}

\begin{assumption}[Completely monotone kernels and finite mass]
	\label{ass:kernels}
	For each $z\in\mathcal Z$, the kernel $g_z$ is completely monotone and belongs to
	$L^1(0,\infty)$. Define its total mass
	\begin{equation}
		\label{eq:kernel-mass}
		G_z := \int_{0}^{\infty} g_z(t)\,dt \in (0,\infty),
		\qquad
		G_{\max}:=\max_{z\in\mathcal Z}G_z<\infty.
	\end{equation}
\end{assumption}

\begin{assumption}[Uniform boundedness]
	\label{ass:uniform}
	There exist constants $M_A,M_F<\infty$ such that
	\begin{equation}
		\label{eq:uniform-bounds}
		\|A_z\|\le M_A,
		\qquad
		\|F_z\|_{\mathcal X}\le M_F,
		\qquad
		\forall z\in\mathcal Z.
	\end{equation}
\end{assumption}

\begin{assumption}[Dissipativity]
	\label{ass:dissip}
	$\mathcal B$ generates a contraction semigroup $e^{t\mathcal B}$ on $\mathcal X$:
	\begin{equation}
		\label{eq:contraction}
		\|e^{t\mathcal B}\|\le 1,
		\qquad t\ge 0.
	\end{equation}
\end{assumption}

\begin{assumption}[Strict dissipativity (when required)]
	\label{ass:dissip-strong}
	There exists $\beta>0$ such that
	\begin{equation}
		\label{eq:strict-dissip}
		\langle \mathcal B U, U\rangle_{\mathcal X}\le -\beta\|U\|_{\mathcal X}^2,
		\qquad \forall U\in\mathcal X.
	\end{equation}
\end{assumption}

\subsection{Memory gain and averaged subcriticality}
\label{subsec:gain}

Define the regime-dependent memory gain
\begin{equation}
	\label{eq:rho-def}
	\rho_z := \|A_z\|\,G_z,
	\qquad
	\rho_{\max}:=\max_{z\in\mathcal Z}\rho_z<\infty,
\end{equation}
and its stationary average
\begin{equation}
	\label{eq:rho-bar}
	\bar\rho := \sum_{z\in\mathcal Z}\pi_z\,\rho_z.
\end{equation}

\subsection{Bernstein representation and auxiliary memory variables}
\label{subsec:bernstein}

Assumption~\ref{ass:kernels} implies a Bernstein representation: for each $z\in\mathcal Z$
there exists a positive measure $\nu_z$ on $(0,\infty)$ such that
\begin{equation}
	\label{eq:bernstein-rep}
	g_z(t)=\int_{(0,\infty)} e^{-rt}\,\nu_z(dr),
	\qquad t>0.
\end{equation}

\section{Pathwise well-posedness}
\label{sec:wp}

This section proves that \eqref{eq:main-switching} admits a unique mild solution
along each sample path of the regime process. The key point is that, conditional
on a fixed environment path, the problem reduces to a deterministic Volterra
equation with piecewise-constant coefficients in time, for which a Picard--Volterra
fixed point argument applies.

\subsection{Pathwise mild solutions}
\label{subsec:mild}

\begin{definition}[Mild solution (pathwise)]
	\label{def:mild}
	Fix $\omega$ and write $z(t)=Z(t,\omega)$. A function
	$U(\cdot,\omega)\in C([0,T];\mathcal X)$ is a mild solution of
	\eqref{eq:main-switching} on $[0,T]$ if for all $t\in[0,T]$,
	\begin{equation}
		\label{eq:mild}
		U(t)
		=
		e^{t\mathcal{B}}U_0
		+
		\int_0^t e^{(t-s)\mathcal{B}}
		\left(
		\int_0^s \mathcal{G}_{z(s)}(s-\tau)U(\tau)\,d\tau
		\right) ds
		+
		\int_0^t e^{(t-s)\mathcal{B}}F_{z(s)}\,ds .
	\end{equation}
\end{definition}

\subsection{Existence and uniqueness}

\begin{theorem}[Pathwise existence and uniqueness]
	\label{thm:pathwise-wp}
	Under Assumptions~\ref{ass:kernels}--\ref{ass:dissip}, for every $T>0$ and for
	every sample path of $Z(\cdot)$, there exists a unique mild solution
	$U\in C([0,T];\mathcal{X})$ of \eqref{eq:main-switching}.
	Moreover, the solution depends continuously on $U_0$ in $C([0,T];\mathcal X)$.
\end{theorem}

\begin{proof}
	Fix $T>0$ and a realisation $\omega$; write $z(t)=Z(t,\omega)$.
	Define the Picard map $\Phi:C([0,T];\mathcal X)\to C([0,T];\mathcal X)$ by the right-hand
	side of \eqref{eq:mild}. Using $\|e^{t\mathcal B}\|\le 1$ from \eqref{eq:contraction} and the
	structure \eqref{eq:Gz-structure}, for any $U\in C([0,T];\mathcal X)$ and $t\in[0,T]$,
	\[
	\left\|\int_0^s \mathcal G_{z(s)}(s-\tau)U(\tau)\,d\tau\right\|
	\le
	\|A_{z(s)}\|\int_0^s g_{z(s)}(s-\tau)\,d\tau \;\sup_{0\le r\le s}\|U(r)\|
	\le
	M_A\,G_{\max}\,\sup_{0\le r\le s}\|U(r)\|.
	\]
	Therefore, for $U,V\in C([0,T];\mathcal X)$,
	\[
	\|(\Phi U)(t)-(\Phi V)(t)\|
	\le
	\int_0^t \left\|\int_0^s \mathcal G_{z(s)}(s-\tau)\big(U(\tau)-V(\tau)\big)\,d\tau\right\| ds
	\le
	M_A G_{\max}\int_0^t \sup_{0\le r\le s}\|U(r)-V(r)\|\,ds,
	\]
	so that
	\[
	\|\Phi U-\Phi V\|_{C([0,T])}
	\le
	M_A G_{\max}\,T\,\|U-V\|_{C([0,T])}.
	\]
	Choose $T_0>0$ such that $M_A G_{\max}T_0<1$. Then $\Phi$ is a contraction on
	$C([0,T_0];\mathcal X)$ and admits a unique fixed point, which is the unique mild
	solution on $[0,T_0]$. Iterating the argument on successive subintervals of length
	$T_0$ yields existence and uniqueness on $[0,T]$.
	
	Continuous dependence on $U_0$ follows from the same contraction estimate applied to
	two solutions with different initial data.
\end{proof}

\begin{remark}[Continuity across regime switches]
	Although $Z(t)$ has jumps, mild solutions $U(t)$ remain continuous in time: switches
	enter only through the integrands in \eqref{eq:mild} and do not create discontinuities.
\end{remark}

\section{Annealed stability via averaged memory gain}
\label{sec:mean-stab}

This section analyses stability in expectation for \eqref{eq:main-switching}.
We derive a Volterra-type a priori inequality, finite-horizon moment bounds,
and a memory Lyapunov functional yielding mean-square dissipation under an averaged
gain margin.

Throughout we work under Assumptions~\ref{ass:kernels}--\ref{ass:dissip} and write
$\|\cdot\|$ for $\|\cdot\|_{\mathcal X}$.

\subsection{A Volterra-type a priori bound}
\label{subsec:volterra-ineq}

\begin{lemma}[A priori inequality]
	\label{lem:volterra-ineq}
	Let $U$ be a mild solution of \eqref{eq:main-switching} on $[0,T]$.
	Then for all $t\in[0,T]$,
	\begin{equation}
		\label{eq:volterra-bound}
		\|U(t)\|
		\le
		\|U_0\| + t\,M_F
		+\int_{0}^{t}\rho_{Z(s)}\,\sup_{0\le \tau\le s}\|U(\tau)\|\,ds.
	\end{equation}
\end{lemma}

\begin{proof}
	Start from \eqref{eq:mild}. By \eqref{eq:contraction}, $\|e^{t\mathcal B}\|\le 1$.
	Using \eqref{eq:Gz-structure},
	\[
	\left\|\int_{0}^{s}\mathcal G_{Z(s)}(s-\tau)\,U(\tau)\,d\tau\right\|
	\le
	\|A_{Z(s)}\|\Big(\int_0^{s} g_{Z(s)}(u)\,du\Big)\sup_{0\le r\le s}\|U(r)\|
	\le
	\rho_{Z(s)}\,\sup_{0\le r\le s}\|U(r)\|.
	\]
	The forcing term satisfies $\left\|\int_0^t e^{(t-s)\mathcal B}F_{Z(s)}ds\right\|\le tM_F$.
	Combining yields \eqref{eq:volterra-bound}.
\end{proof}

\subsection{Finite-horizon annealed bounds}
\label{subsec:annealed-bounds}

Define $Y(t):=\sup_{0\le s\le t}\|U(s)\|$.
From Lemma~\ref{lem:volterra-ineq},
\begin{equation}
	\label{eq:Y-ineq}
	Y(t)
	\le
	\|U_0\| + t\,M_F + \int_0^t \rho_{Z(s)}\,Y(s)\,ds.
\end{equation}

\begin{theorem}[Finite-horizon moment bound]
	\label{thm:annealed-finiteT}
	Under Assumptions~\ref{ass:kernels}--\ref{ass:dissip}, for every $T>0$ there exists
	$C_T<\infty$ such that
	\begin{equation}
		\label{eq:annealed-finiteT}
		\sup_{0\le t\le T}\mathbb E\|U(t)\|
		\le
		C_T\big(\|U_0\|+M_F\big).
	\end{equation}
\end{theorem}

\begin{proof}
	From \eqref{eq:Y-ineq} we obtain the pathwise bound
	\[
	Y(t)\le (\|U_0\|+tM_F)\exp\!\Big(\int_0^t \rho_{Z(s)}\,ds\Big)
	\le (\|U_0\|+TM_F)\,e^{\rho_{\max}t},
	\qquad t\in[0,T].
	\]
	Taking expectations yields \eqref{eq:annealed-finiteT}.
\end{proof}

\subsection{A Lyapunov functional with memory and mean-square dissipation}
\label{subsec:lyap}

Using \eqref{eq:bernstein-rep}, along a fixed environment path $z(\cdot)=Z(\cdot,\omega)$ define
\begin{equation}
	\label{eq:aux-memory}
	w(t,r):=\int_0^t e^{-r(t-s)}\,U(s)\,ds,
	\qquad r>0,
\end{equation}
so that $\partial_t w(t,r)=U(t)-r w(t,r)$ and $w(0,r)=0$.

For $\eta>0$ define
\begin{equation}
	\label{eq:lyapunov}
	\mathcal V_{z}(t)
	:=
	\|U(t)\|^2
	+
	\eta\int_{(0,\infty)} r\,\|w(t,r)\|^2\,\nu_{z}(dr),
	\qquad z\in\mathcal Z.
\end{equation}

\begin{theorem}[Mean-square dissipation under an averaged margin]
	\label{thm:lyapunov-mean-stability}
	Assume $F_z\equiv 0$ and Assumption~\ref{ass:dissip-strong} holds with constant $\beta>0$.
	Let $\bar\rho=\sum_{z\in\mathcal Z}\pi_z\rho_z$.
	If $\bar\rho<\beta$, then there exist $\eta>0$ and $c>0$ such that
	\begin{equation}
		\label{eq:lyapunov-dissipation}
		\frac{d}{dt}\,\mathbb E\!\left[\mathcal V_{Z(t)}(t)\right]
		\le
		-\,c\,\mathbb E\|U(t)\|^2,
		\qquad t\ge 0.
	\end{equation}
	In particular,
	\begin{equation}
		\label{eq:mean-square-bound}
		\sup_{t\ge 0}\mathbb E\|U(t)\|^2<\infty
		\qquad\text{and}\qquad
		\int_0^\infty \mathbb E\|U(t)\|^2\,dt <\infty.
	\end{equation}
\end{theorem}

\begin{proof}[Proof outline]
	Fix a sample path of $Z(\cdot)$ and differentiate $\|U(t)\|^2$ using \eqref{eq:main-switching}
	with $F\equiv 0$, together with $\partial_t w(t,r)=U(t)-r w(t,r)$.
	Apply strict dissipativity \eqref{eq:strict-dissip} and control coupling terms via Young inequalities.
	The Bernstein representation implies $\int_{(0,\infty)} r^{-1}\nu_z(dr)=G_z$, linking the memory drift to
	$\rho_{Z(t)}=\|A_{Z(t)}\|G_{Z(t)}$. Taking expectations and using stationarity under $\pi$ yields
	\eqref{eq:lyapunov-dissipation} provided $\bar\rho<\beta$ after choosing $\eta$ appropriately.
\end{proof}

\begin{remark}[Forced case]
	If $F_z\not\equiv 0$, the same calculation yields
	\[
	\frac{d}{dt}\,\mathbb E[\mathcal V_{Z(t)}(t)]
	\le
	-\,c\,\mathbb E\|U(t)\|^2
	+
	C\,\sup_{z\in\mathcal Z}\|F_z\|^2,
	\]
	for suitable constants $c,C>0$, implying uniform-in-time boundedness of $\mathbb E\|U(t)\|^2$.
\end{remark}

\section{Quenched growth, intermittency, and an annealed--quenched dichotomy}
\label{sec:as-stab}

Section~\ref{sec:mean-stab} provides annealed control under averaged gain margins.
This section shows that such control does not characterise typical sample paths when long
memory interacts with random residence times.

Throughout we set $F\equiv 0$.

\subsection{Annealed versus quenched notions}
\label{subsec:annealed-quenched}

Let $U(t,\omega)$ be the mild solution of \eqref{eq:main-switching}.

\begin{definition}[Annealed stability]
	\label{def:annealed}
	The system \eqref{eq:main-switching} is \emph{annealed stable} if
	$\sup_{t\ge 0}\mathbb E\|U(t)\|<\infty$.
\end{definition}

\begin{definition}[Quenched (pathwise) growth exponent]
	\label{def:quenched-exp}
	Define
	\begin{equation}
		\label{eq:Lambda}
		\Lambda(\omega):=\limsup_{t\to\infty}\frac{1}{t}\log\|U(t,\omega)\|.
	\end{equation}
\end{definition}

\begin{definition}[Intermittency]
	\label{def:intermittency}
	The system exhibits \emph{intermittency} if it is annealed stable and
	$\mathbb P(\Lambda(\omega)>0)>0$.
\end{definition}

\subsection{Rare-event amplification in a two-regime switching model}
\label{subsec:rare-events}

Consider $\mathcal Z=\{S,U\}$ with transition rates $q_{SU},q_{US}>0$ and
$\tau_U\sim\mathrm{Exp}(q_{US})$ the sojourn time in $U$.

\begin{assumption}[Frozen-regime exponential bound]
	\label{ass:frozen-exp}
	Let $R_z(t)$ denote the frozen resolvent family associated with \eqref{eq:main-switching}
	when $Z(t)\equiv z$. There exist $C_z\ge1$ and $\gamma_z\in\mathbb R$ such that
	$\|R_z(t)\|\le C_z e^{\gamma_z t}$ for $t\ge0$, with $\gamma_S<0<\gamma_U$.
\end{assumption}

\begin{theorem}[Burst amplification and heavy-tailed burst sizes]
	\label{thm:burst-qUS}
	Assume the two-regime setting and suppose that during any visit to $U$,
	$\|U(t)\|\ge c_0 e^{\gamma_U t}\|U(0)\|$ on $[0,\tau_U]$ for some $c_0\in(0,1]$.
	Define $B:=\exp(\gamma_U\tau_U)$. Then:
	\begin{enumerate}[label=(\roman*),leftmargin=0.8cm]
		\item $\mathbb E[B^p]<\infty$ iff $p\gamma_U<q_{US}$.
		\item For $b>1$,
		\begin{equation}
			\label{eq:burst-tail}
			\mathbb P(B>b)=b^{-q_{US}/\gamma_U}.
		\end{equation}
	\end{enumerate}
\end{theorem}

\subsection{A subadditive structure and deterministic quenched exponents}
\label{subsec:subadditive}

Write $U(t,\omega)=\mathcal U(t,\omega)U_0$ and set
$\Lambda_{\mathcal U}(\omega):=\limsup_{t\to\infty}t^{-1}\log\|\mathcal U(t,\omega)\|$.

\begin{theorem}[Deterministic quenched exponent]
	\label{thm:subadditive}
	Assume $Z(t)$ is stationary and ergodic and that
	$\mathbb E[\log^+\|\mathcal U(1,\cdot)\|]<\infty$.
	Then there exists a deterministic $\gamma\in[-\infty,\infty)$ such that
	$\Lambda_{\mathcal U}(\omega)=\gamma$ for $\mathbb P$-almost every $\omega$ and
	$\Lambda(\omega)\le \gamma$ a.s.
\end{theorem}

\subsection{An intermittency window}
\label{subsec:intermittency-window}
\label{subsec:intermittency-defs}


\begin{proposition}[Concrete intermittency window (two-regime illustration)]
	\label{prop:intermittency-window-2reg}
	Assume $\mathcal Z=\{S,U\}$ and let
	$\pi_U=\frac{q_{SU}}{q_{SU}+q_{US}}$, $\pi_S=\frac{q_{US}}{q_{SU}+q_{US}}$.
	Let $\bar\rho=\pi_S\rho_S+\pi_U\rho_U$ and assume $\bar\rho<\beta$ (Theorem~\ref{thm:lyapunov-mean-stability}).
	If $q_{US}\le p\gamma_U$ for some $p>0$, then $B$ has infinite $p$-th moment and power-law tail
	\eqref{eq:burst-tail}; consequently annealed stability can coexist with quenched burst amplification
	(intermittency).
\end{proposition}

\section{Micro--macro limits for Hawkes processes in a switching environment}
\label{sec:micro-macro}
\label{sec:micro-macro-random} 

This section connects \eqref{eq:main-switching} to microscopic branching dynamics via
regime-modulated Hawkes processes.

\subsection{Microscopic model: network Hawkes with regime-dependent kernels}
\label{subsec:hawkes-model}

Let $Z(t)$ be ergodic on $\mathcal Z$ and independent of the Hawkes driving noise.
Fix $n\in\mathbb N$ and, for each $N\in\mathbb N$, consider $N$ i.i.d.\ replicas of an
$n$-dimensional Hawkes process with intensities
\begin{equation}
	\label{eq:hawkes-random-intensity}
	\lambda^{i,k}_N(t)
	=
	\mu_{Z(t),i}
	+
	\sum_{j=1}^n A_{Z(t),ij}\int_{(0,t)} g_{Z(t)}(t-s)\,dN^{j,k}(s),
	\qquad t\ge 0.
\end{equation}
Assume uniform bounds $\sup_z\|A_z\|<\infty$, $\sup_z\|\mu_z\|<\infty$, $\sup_z G_z<\infty$.

Define
\begin{equation}
	\label{eq:empirical-intensity}
	\bar\lambda_N(t):=\frac1N\sum_{k=1}^N \lambda_N^{(\cdot),k}(t),
	\qquad
	\bar N_N(t):=\frac1N\sum_{k=1}^N N^{(\cdot),k}(t).
\end{equation}

\subsection{Macroscopic limit: Volterra equation with random coefficients}
\label{subsec:macro-volterra}

Given $Z(\cdot)$, define the macroscopic Volterra equation
\begin{equation}
	\label{eq:macro-random-volterra}
	\lambda(t)
	=
	\mu_{Z(t)}
	+
	A_{Z(t)}\int_0^t g_{Z(t)}(t-s)\,\lambda(s)\,ds,
	\qquad t\ge 0.
\end{equation}

\subsection{Annealed hydrodynamic limit}
\label{subsec:annealed-limit}

\begin{theorem}[Annealed micro--macro limit]
	\label{thm:annealed-hawkes}
	Under the standing bounds and independence of $Z(\cdot)$ from the Hawkes noise,
	for every $T>0$,
	\[
	\sup_{0\le t\le T}\|\bar\lambda_N(t)-\lambda(t)\|
	\ \xrightarrow[N\to\infty]{\mathbb P}\ 0,
	\]
	where $\lambda$ solves \eqref{eq:macro-random-volterra}.
\end{theorem}

\subsection{Quenched limit conditional on the environment}
\label{subsec:quenched-limit}

\begin{theorem}[Quenched micro--macro limit]
	\label{thm:quenched-hawkes}
	Under the assumptions of Theorem~\ref{thm:annealed-hawkes}, there exists $\Omega_0$ with
	$\mathbb P(\Omega_0)=1$ such that for every $\omega\in\Omega_0$ and every $T>0$,
	\[
	\sup_{0\le t\le T}\|\bar\lambda_N(t,\omega)-\lambda(t,\omega)\|
	\ \xrightarrow[N\to\infty]{}\ 0,
	\]
	where $\lambda(\cdot,\omega)$ solves \eqref{eq:macro-random-volterra} with coefficients frozen along
	$Z(\cdot,\omega)$.
\end{theorem}

\subsection{Inheritance of burst amplification}
\label{subsec:inheritance}

\begin{corollary}[Inheritance of heavy-tailed burst multipliers]
	\label{cor:burst-inheritance}
	Consider $\mathcal Z=\{S,U\}$ and assume the frozen macroscopic dynamics in $U$ admits
	growth rate $\gamma_U>0$ (Assumption~\ref{ass:frozen-exp}). Let $\tau_U\sim\mathrm{Exp}(q_{US})$
	and $B=\exp(\gamma_U\tau_U)$. Then $\mathbb P(B>b)=b^{-q_{US}/\gamma_U}$ for $b>1$, and for large $N$
	the empirical mean intensity exhibits burst events with the same environment-driven tail behaviour
	conditional on typical environment paths.
\end{corollary}

\section{Numerical experiments}
\label{sec:numerics}

This section provides systematic numerical evidence for the mechanisms developed in
Sections~\ref{sec:mean-stab}--\ref{sec:micro-macro}.
The experimental design is explicitly aligned with the theory:
(i) \emph{annealed} control is assessed through empirical moment proxies consistent with
Theorem~\ref{thm:lyapunov-mean-stability};
(ii) \emph{quenched} amplification is assessed through pathwise growth proxies and burst-tail diagnostics consistent with
Theorems~\ref{thm:burst-qUS} and~\ref{thm:subadditive};
(iii) \emph{intermittency} is diagnosed by the joint presence of controlled annealed summaries and heavy-tailed quenched bursts,
as formalized in Section~\ref{subsec:intermittency-defs}; and
(iv) the \emph{geometric organization} of bursts is quantified via spectral routing and localization metrics (Experiments~IV--V).
Finally, we validate the micro--macro correspondence for regime-driven Hawkes processes (Experiment~VI), supporting
Theorems~\ref{thm:annealed-hawkes}--\ref{thm:quenched-hawkes}.

All simulations are fully reproducible.
For each Monte Carlo realisation we store: the complete regime path, state trajectories, modewise projections, and burst diagnostics
(peak time, peak amplitude, burst direction, CCDF samples).
Figures are generated in a separate post-processing stage to ensure that visualization choices do not affect solver behaviour.
Random seeds are fixed and reported.

\subsection{Model discretization, kernel approximation, and switching simulation}
\label{subsec:num-model}

We simulate the switching Volterra system \eqref{eq:main-switching} on a uniform temporal grid
$t_n=n\Delta t$, $n=0,\dots,N_T$, with $T=N_T\Delta t$.
Given the active regime $z_n:=Z(t_n)$, we approximate the Volterra convolution by a causal Riemann sum:
\begin{equation}
	\int_0^{t_n} g_{z_n}(t_n-s)\,y(s)\,ds
	\;\approx\;
	\Delta t \sum_{k=0}^{n-1} g_{z_n}\!\big((n-k)\Delta t\big)\,y_k.
	\label{eq:num-riemann}
\end{equation}
This discretization is deliberately causal and uses the regime evaluated at the observation time $t_n$, matching the structure of the
random-coefficient equation.

To ensure efficiency for slowly decaying kernels, we approximate each completely monotone kernel by a positive SOE:
\begin{equation}
	g_z(t)\approx \sum_{\ell=1}^{K} w_{z,\ell}\,e^{-r_{z,\ell}t},
	\qquad w_{z,\ell}>0,\ r_{z,\ell}>0,
	\label{eq:num-soe}
\end{equation}
enabling recursive updates with $\mathcal O(K)$ memory per time step.
For each $\ell$, define the discrete auxiliary states
\[
s^{(\ell)}_{n} \approx \int_0^{t_n} e^{-r_{z_n,\ell}(t_n-s)}y(s)\,ds,
\]
updated by a one-step recursion induced by \eqref{eq:num-riemann}--\eqref{eq:num-soe}.
Unless otherwise stated we use $K\in[18,24]$ and verify robustness by increasing $K$ (which sharpens burst-tail estimates but does not
change qualitative phase boundaries).

The full state $U_n=(x_n,\lambda_n)\in\mathcal X$ is evolved by a semi-implicit linear scheme, implicit in the dissipative part and
explicit in the Volterra feedback:
\begin{equation}
	(I-\Delta t\,\mathcal B)\,U_{n+1}
	=
	U_n+\Delta t\,\mathcal C_{z_n}[U_{0:n}],
	\label{eq:num-semi-imp}
\end{equation}
where $\mathcal C_{z_n}[U_{0:n}]$ is the regime-dependent discrete Volterra operator induced by
\eqref{eq:num-riemann}--\eqref{eq:num-soe}.
This isolates, in a numerically controlled way, the burst-generating mechanism produced by memory plus switching, while removing
stability restrictions associated with stiff dissipation.
In burst-dominated regimes we verified that replacing \eqref{eq:num-semi-imp} by fully explicit updates yields the same qualitative
burst statistics provided $\Delta t$ is sufficiently small.

The regime process $Z(t)$ is simulated as a continuous-time Markov chain using \emph{exact exponential holding times}.
We sample jump times, then assign $z_n$ by evaluating the (piecewise constant) chain at $t_n$.
This is essential: burst tails depend sensitively on rare long residence times in unstable regimes, and discretizations that distort
holding-time statistics bias the tail diagnostics and can destroy Theorem~\ref{thm:burst-qUS}-type scaling.

Across experiments we additionally perform: (i) time-step refinement ($\Delta t\downarrow$) checks for representative parameter sets;
(ii) SOE refinement ($K\uparrow$) checks for heavy-tail diagnostics; and (iii) horizon checks ($T\uparrow$) to separate transient growth
from persistent intermittency. All reported qualitative conclusions are stable under these refinements.

\subsection{Experimental setup, observables, and diagnostics}
\label{subsec:num-observables}

For each Monte Carlo realisation we record:
\begin{itemize}
	\item the norm trajectory $t\mapsto \|U(t)\|_{\mathcal X}$;
	\item the peak time $t^*\in\arg\max_{t\in[0,T]}\|U(t)\|_{\mathcal X}$ and burst size
	\begin{equation}
		B:=\max_{t\le T}\frac{\|U(t)\|_{\mathcal X}}{\|U(0)\|_{\mathcal X}};
		\label{eq:num-burst}
	\end{equation}
	\item a finite-horizon growth proxy
	\begin{equation}
		\gamma_T:=\frac{1}{T}\log\!\left(\frac{\|U(T)\|_{\mathcal X}}{\|U(0)\|_{\mathcal X}}\right),
		\label{eq:num-gammat}
	\end{equation}
	used as a numerical surrogate for the almost sure exponent of Section~\ref{subsec:subadditive};
	\item the burst direction $z^*:=U(t^*)/\|U(t^*)\|_{\mathcal X}$.
\end{itemize}

When analysing network structure, we compute (i) projections onto graph Laplacian eigenmodes to quantify spectral routing at $t^*$,
and (ii) node-basis localization via the inverse participation ratio (IPR),
$\mathrm{IPR}(z^*)=\sum_i (z_i^*)^4$ (after normalizing $z^*$ in $\ell^2$).

Annealed control is assessed using empirical analogues of
$\sup_{t\le T}\mathbb E\|U(t)\|_{\mathcal X}$ and, when appropriate,
$\sup_{t\le T}\mathbb E\|U(t)\|_{\mathcal X}^2$.
Quenched amplification is assessed via:
(i) the tail of $B$ (empirical CCDF on log--log axes), and
(ii) the distribution of $\gamma_T$ (median, upper quantiles, and sign frequency).
Intermittency is diagnosed by the joint occurrence of: bounded annealed summaries on $[0,T]$ and nontrivial burst probability with
positive typical growth proxies, consistent with Section~\ref{subsec:intermittency-defs}.

To connect to Theorem~\ref{thm:burst-qUS}, we additionally estimate tail slopes from the CCDF over a data-driven scaling window
(e.g.\ via linear regression on log--log CCDF segments and a Hill-type estimator on the largest order statistics), and report stability
of the estimated exponent under $T$ and $K$ refinements.

\subsection{Baseline parameter set}
\label{subsec:num-baseline}

Table~\ref{tab:num-params} summarizes the baseline parameters used throughout the experiments unless otherwise stated.
We focus primarily on two regimes $\mathcal Z=\{S,U\}$ with a subcritical stable regime ($\rho_S<1$) and a supercritical unstable
regime ($\rho_U>1$), ensuring burst amplification is possible under switching.

\begin{table}[H]
	\centering
	\caption{Baseline numerical parameters used across experiments (unless otherwise stated).}
	\label{tab:num-params}
	\begin{tabular}{@{}lll@{}}
		\toprule
		Category & Quantity & Value / description \\
		\midrule
		Time discretization
		& Time step $\Delta t$
		& $0.05$ \\
		& Horizon $T$
		& $120$ \\
		\midrule
		Regime switching
		& Switching rates $(q_{SU},q_{US})$
		& $(0.08,\;0.008)$ \\
		& Number of regimes
		& $2$ (stable $S$ / unstable $U$) \\
		\midrule
		Memory kernel
		& Kernel form $g_z(t)$
		& $t^{-\alpha} e^{-\theta t}/\Gamma(1-\alpha)$ \\
		& Fractional order $\alpha$
		& $0.65$--$0.90$ (experiment dependent) \\
		& Tempering $\theta$
		& $0$--$0.35$ (experiment dependent) \\
		& SOE terms $K$
		& $18$--$24$ \\
		\midrule
		Dynamics
		& Linear damping margin $\beta$
		& $1.5$--$3.5$ \\
		& Relative burst threshold $B_{\mathrm{rel}}$
		& $10$--$200$ (experiment dependent) \\
		\midrule
		Network structure
		& Graph geometries
		& ring, star, Erd\H{o}s--R\'enyi, small-world \\
		& Network size $n$
		& $40$--$240$ (experiment dependent) \\
		\midrule
		Monte Carlo
		& Number of paths $N_{\mathrm{paths}}$
		& $80$--$250$ \\
		& Random seeds
		& fixed and reported \\
		\bottomrule
	\end{tabular}
\end{table}

\subsection{Experiment I: Two-regime switching and burst amplification}
\label{subsec:num-exp1}
We first isolate the burst mechanism of Section~\ref{subsec:rare-events}:
rare but persistent excursions into the unstable regime amplified by long memory.
Figure~\ref{fig:exp1} reports:
(A) a representative trajectory $\|U(t)\|_{\mathcal X}$ with regime intervals indicated,
(B) modewise projections at/near the burst peak, and
(C) the empirical CCDF of burst sizes $B$ on log--log axes.

\begin{figure}[t]
	\centering
	\includegraphics[width=1.0\textwidth]{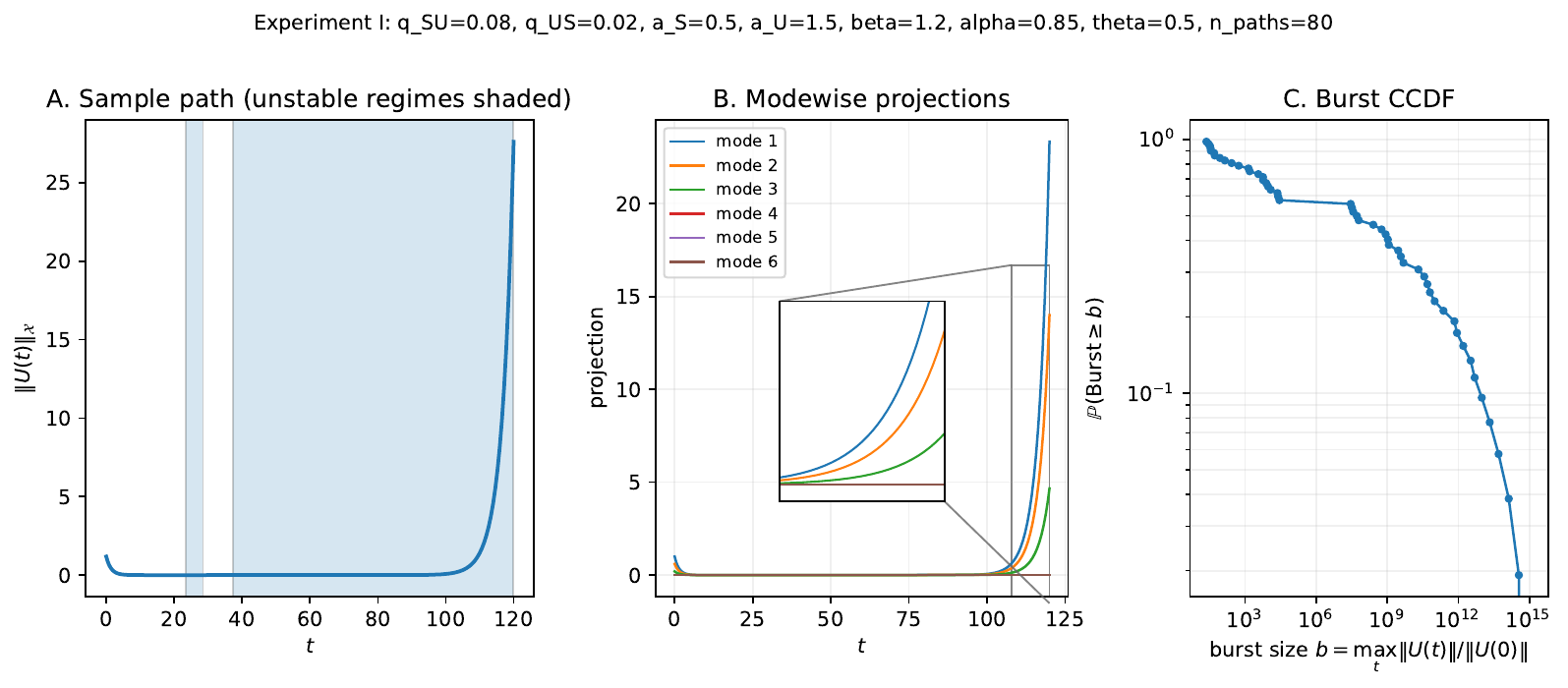}
	\caption{
		\textbf{Experiment I: burst amplification under two-regime switching.}
		(A) Representative sample path of $\|U(t)\|_{\mathcal X}$ with unstable regime intervals shaded.
		(B) Modewise projections onto dominant Laplacian eigenmodes, showing that burst amplification is spectrally concentrated.
		(C) Empirical CCDF of burst sizes $B$, exhibiting an approximate power-law region consistent with
		Theorem~\ref{thm:burst-qUS}.
	}
	\label{fig:exp1}
\end{figure}

A typical trajectory remains quiescent for long periods, consistent with bounded annealed means on $[0,T]$,
before a single unusually long sojourn in the unstable regime produces an abrupt multi-decade increase in $\|U(t)\|_{\mathcal X}$.
This is the finite-horizon numerical counterpart of the rare-event mechanism formalized in Section~\ref{subsec:rare-events}.

The CCDF of $B$ displays an approximate power-law over several decades.
Finite-horizon effects, multiple visits to $U$, and discretization produce deviations from a pure asymptotic tail, but the observed scaling
is consistent with Theorem~\ref{thm:burst-qUS}: residence-time statistics translate directly into burst-tail exponents.

\subsection{Experiment II: Memory parameters and the intermittency window}
\label{subsec:num-exp2}

We sweep $(\alpha,\theta)$ in the fractional-tempered kernel
$g(t)=t^{-\alpha}e^{-\theta t}/\Gamma(1-\alpha)$ while keeping switching rates and regime gains fixed.
Figure~\ref{fig:exp2} reports:
(A) annealed mean responses,
(B) distributions of $\gamma_T$, and
(C) burst-size CCDFs.

\begin{figure}[t]
	\centering
	\includegraphics[width=0.99\textwidth]{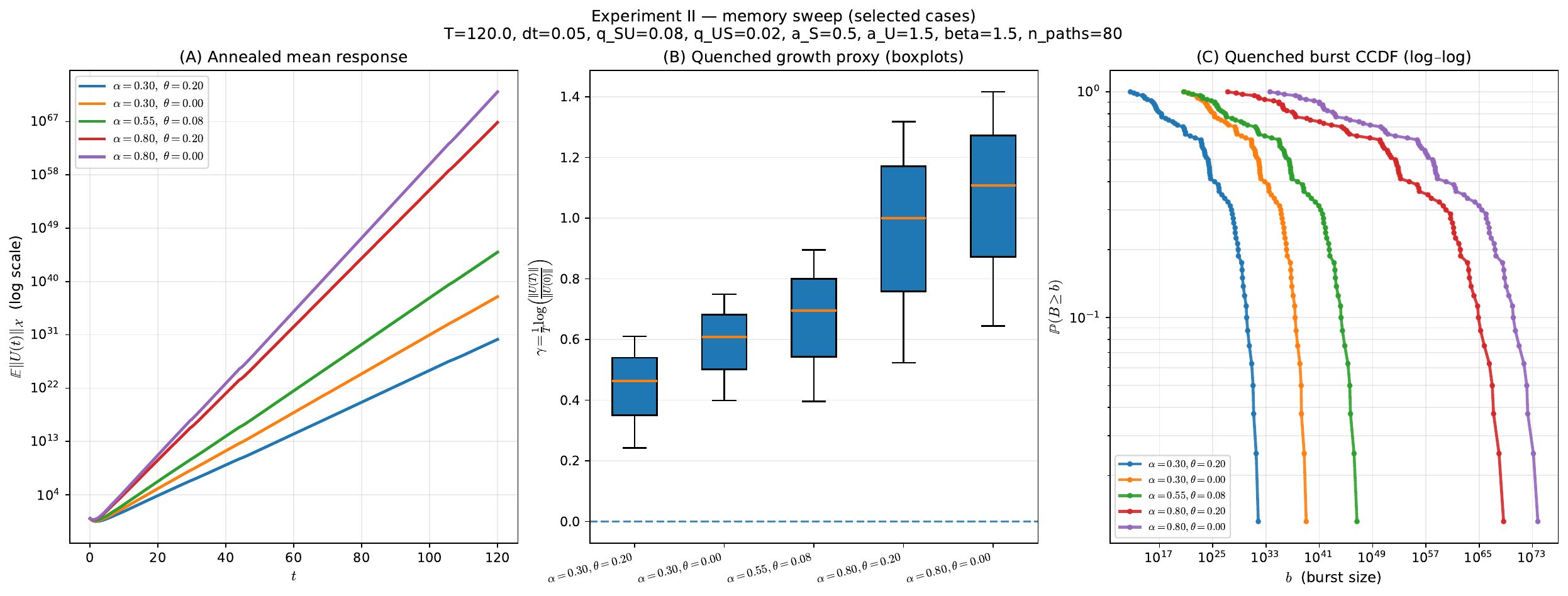}
	\caption{
		\textbf{Experiment II: memory-driven intermittency.}
		Increasing $\alpha$ (stronger memory) and decreasing $\theta$ (weaker tempering)
		broaden burst tails and shift the distribution of $\gamma_T$ toward positive values,
		even when annealed summaries remain moderate on $[0,T]$.
	}
	\label{fig:exp2}
\end{figure}

Annealed summaries increase monotonically with heavier memory (larger $\alpha$, smaller $\theta$),
consistent with the averaged gain viewpoint of Section~\ref{sec:mean-stab}.
Tempering suppresses growth by shortening effective memory and reducing long-sojourn amplification.

Even when annealed responses remain moderate on $[0,T]$, the distribution of $\gamma_T$ often has positive median and heavy upper tail,
supporting the intermittency concept of Section~\ref{subsec:intermittency-defs}.

Burst-size CCDFs broaden significantly as memory strengthens: small $\theta$ and large $\alpha$ produce heavier tails.
Mechanistically, long residence times in $U$ provide the exposure while long memory increases the accumulated unstable feedback.

\subsection{Experiment III: Switching-rate phase diagram (annealed vs.\ quenched)}
\label{subsec:num-exp3}

We scan $(q_{SU},q_{US})$ on a grid and construct a finite-horizon phase diagram using:
(i) an annealed boundedness indicator based on empirical terminal means,
(ii) the burst probability
$P_{\mathrm{burst}}:=\mathbb P(B>B_{\mathrm{rel}})$ estimated across paths, and
(iii) the median growth proxy $\mathrm{Med}[\gamma_T]$.
Figure~\ref{fig:exp3} reports the resulting diagram.

\begin{figure}[t]
	\centering
	\includegraphics[width=1.02\textwidth]{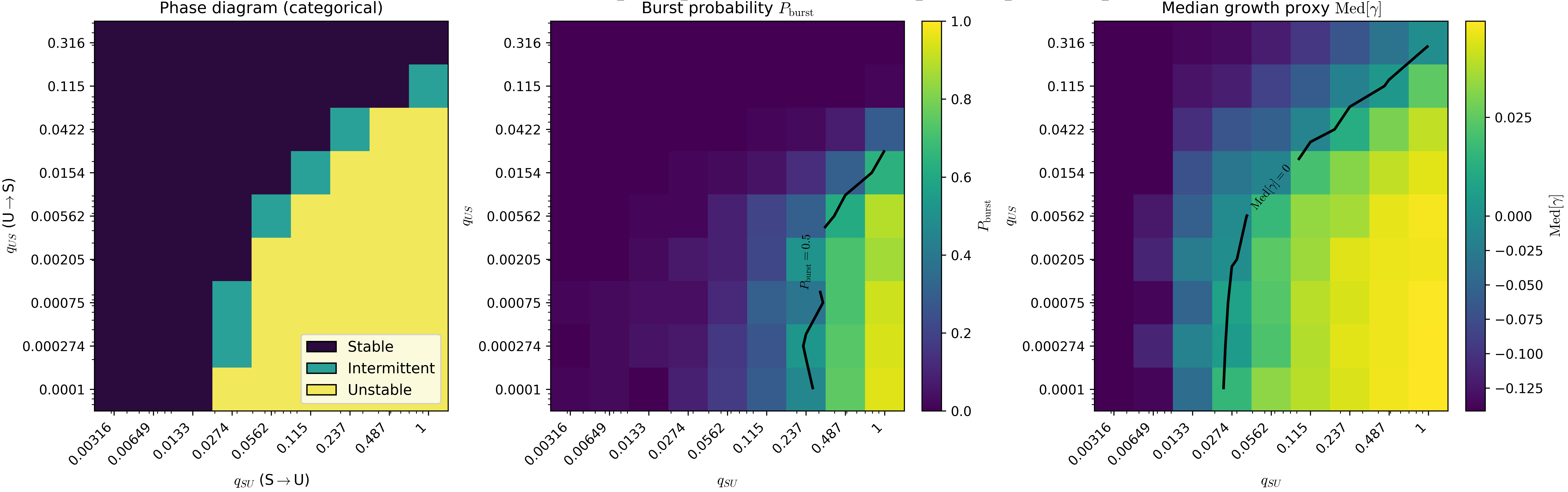}
	\caption{
		\textbf{Experiment III: switching-rate phase diagram (annealed vs.\ quenched signatures).}
		The intermittent band appears where annealed summaries remain controlled while
		$P_{\mathrm{burst}}$ is nontrivial and $\mathrm{Med}[\gamma_T]>0$,
		consistent with Proposition~\ref{prop:intermittency-window-2reg}.
	}
	\label{fig:exp3}
\end{figure}

The diagram exhibits a clear intermittent band separating annealed-stable and fully unstable regions.
Decreasing $q_{US}$ (longer unstable sojourns) increases burst probability and thickens burst tails,
while increasing $q_{SU}$ (more frequent entrances into $U$) increases the number of burst opportunities.
The observed monotone structure is consistent with the residence-time mechanism in Theorem~\ref{thm:burst-qUS}.

\subsection{Experiment IV: Noncommutative excitation operators and mode mixing}
\label{subsec:num-exp4}

To stress-test modal intuition, we construct excitation operators $A_S$ and $A_U$ that do not commute with each other and do not commute
with the graph Laplacian, eliminating invariant eigenspaces.
Figure~\ref{fig:exp4} reports burst-direction alignment statistics, dominant-mode distributions at burst peaks, and burst tails.

\begin{figure}[t]
	\centering
	\includegraphics[width=1.0\textwidth]{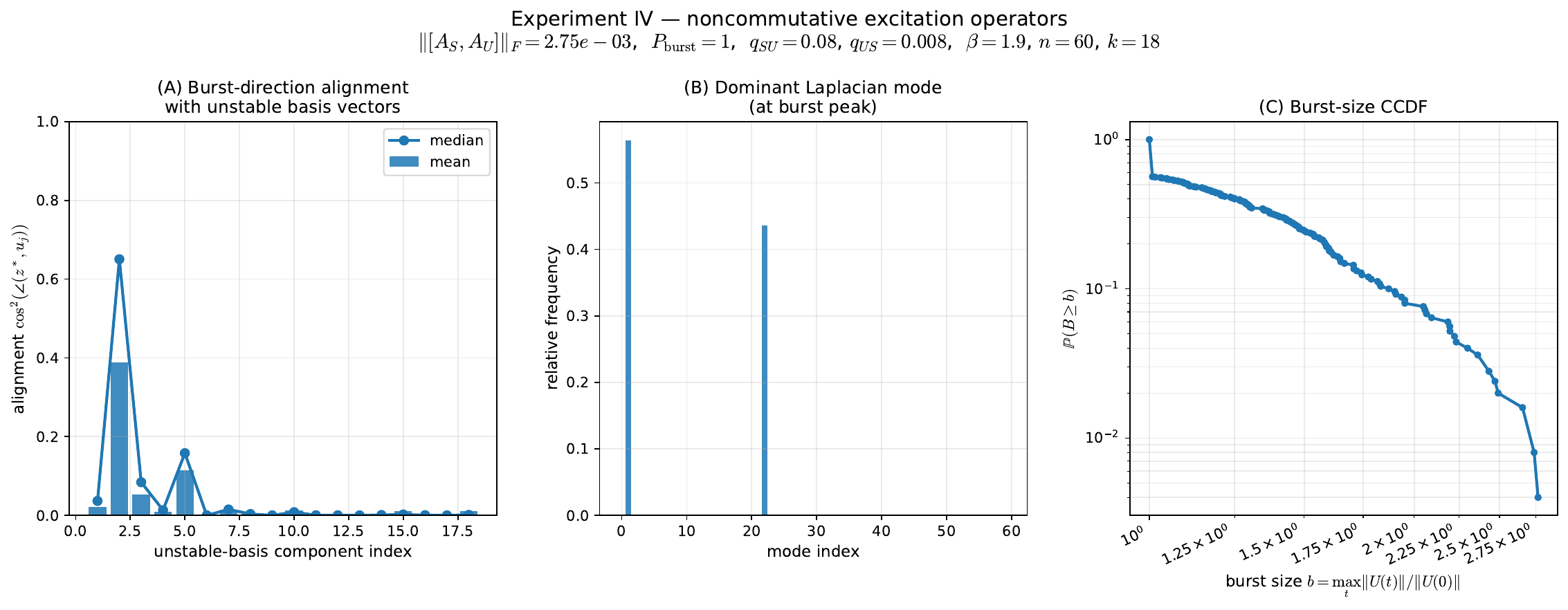}
	\caption{
		\textbf{Experiment IV: noncommutative excitation operators.}
		Noncommutativity induces mode mixing (bursts no longer align with a single Laplacian mode),
		yet burst directions remain concentrated in a low-dimensional subspace and burst tails remain heavy.
	}
	\label{fig:exp4}
\end{figure}

Two robust phenomena emerge.
First, noncommutativity induces \emph{mode mixing}: bursts no longer align with a single Laplacian mode, and energy spreads across a
restricted band of modes repeatedly activated across realizations.
Second, despite this mixing, burst directions remain concentrated in a low-dimensional subspace and burst tails remain heavy.
Thus intermittent amplification is not an artifact of commuting/diagonalizable structure.

\subsection{Experiment V: Network topology, size, and burst localisation}
\label{subsec:num-exp5}

We repeat the dynamics on ring, star, Erd\H{o}s--R\'enyi, and small-world graphs and sweep the network size $n$.
For each path we extract the burst direction $z^*$ and measure spectral routing via Laplacian projections and spatial localisation via
$\mathrm{IPR}(z^*)$.
Figure~\ref{fig:exp5} summarises dominant-band selection and localisation scaling.

\begin{figure}[t]
	\centering
	\includegraphics[width=1.0\textwidth]{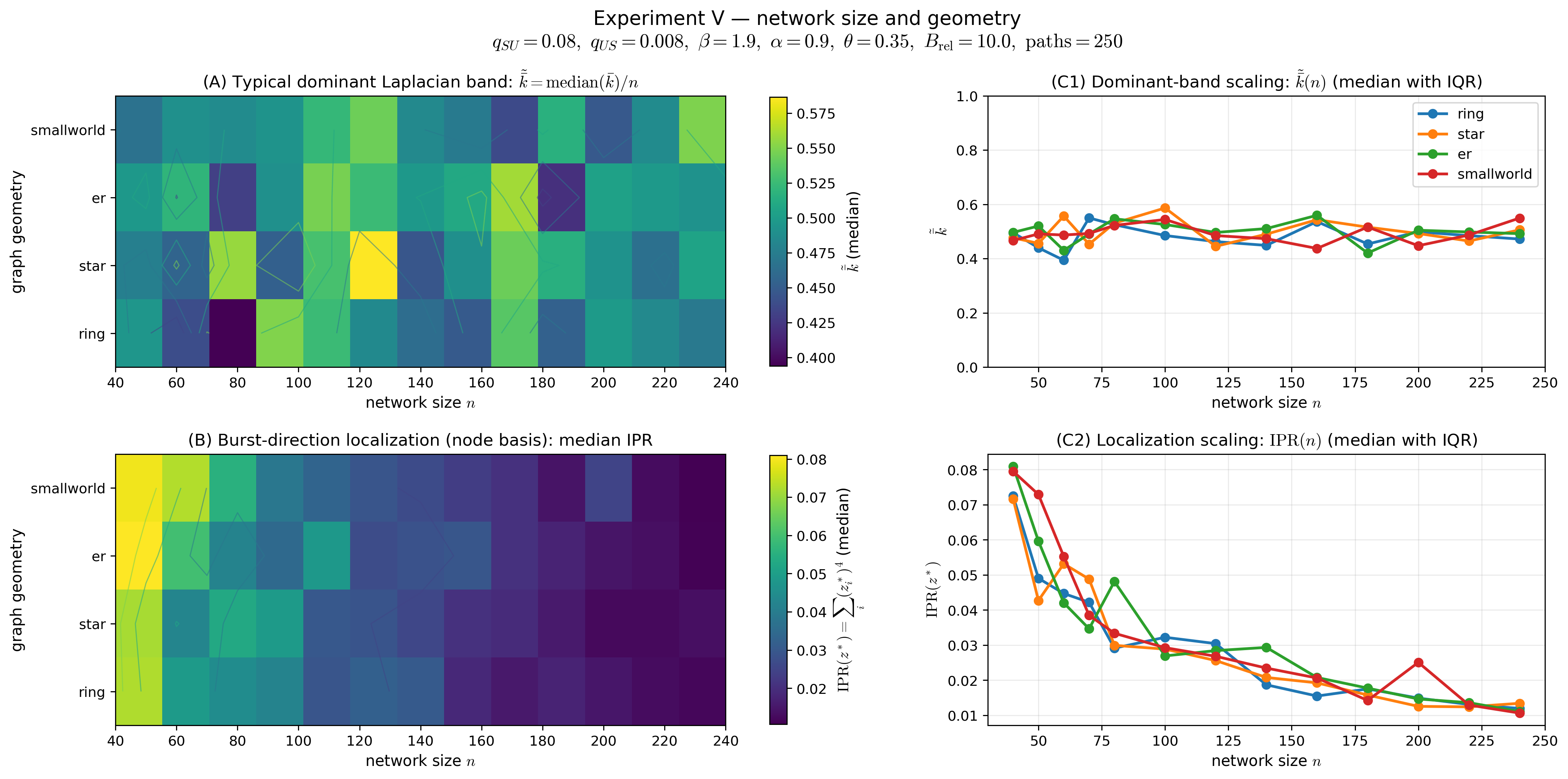}
	\caption{
		\textbf{Experiment V: network size and geometry.}
		Topology primarily modulates where bursts concentrate in node space (IPR scaling),
		while dominant spectral activation typically occurs in intermediate Laplacian bands.
	}
	\label{fig:exp5}
\end{figure}

Across all geometries, dominant spectral activation typically occurs in intermediate bands rather than exclusively in the lowest modes,
indicating that bursts are shaped by the interaction between memory-weighted feedback and instantaneous network dissipation rather than by
a single slow mode.
Localization exhibits a strong and geometry-dependent size law: $\mathrm{IPR}(z^*)$ decreases with $n$ but with systematic offsets across
topologies, reflecting differences in eigenvector structure and degree heterogeneity.

\subsection{Experiment VI: Numerical validation of the micro--macro correspondence under regime switching}
\label{subsec:num-exp6}

We validate the micro--macro limits of Section~\ref{sec:micro-macro-random} by simulating the regime-dependent Hawkes system
\eqref{eq:hawkes-random-intensity} with network size $n$ and $N\in\{10,\dots,800\}$ independent replicas, and comparing the empirical mean
intensity $\bar\lambda_N$ to the Volterra limit $\lambda$ solving \eqref{eq:macro-random-volterra} driven by the same environment path.

To avoid artificial inflation of errors during burst phases, we report the relative finite-horizon discrepancy
\begin{equation}
	\mathrm{Err}^{\mathrm{rel}}_N(T)
	:=
	\sup_{0\le t\le T}
	\frac{\|\bar\lambda_N(t)-\lambda(t)\|}{1+\|\lambda(t)\|}.
	\label{eq:num-errN-rel}
\end{equation}

Figure~\ref{fig:exp6} reports:
(A) quenched convergence along typical environment paths,
(B) annealed convergence averaged over independent environment realisations, and
(C) error trajectories along environment paths producing the largest macroscopic bursts.

\begin{figure}[t]
	\centering
	\includegraphics[width=1.0\textwidth]{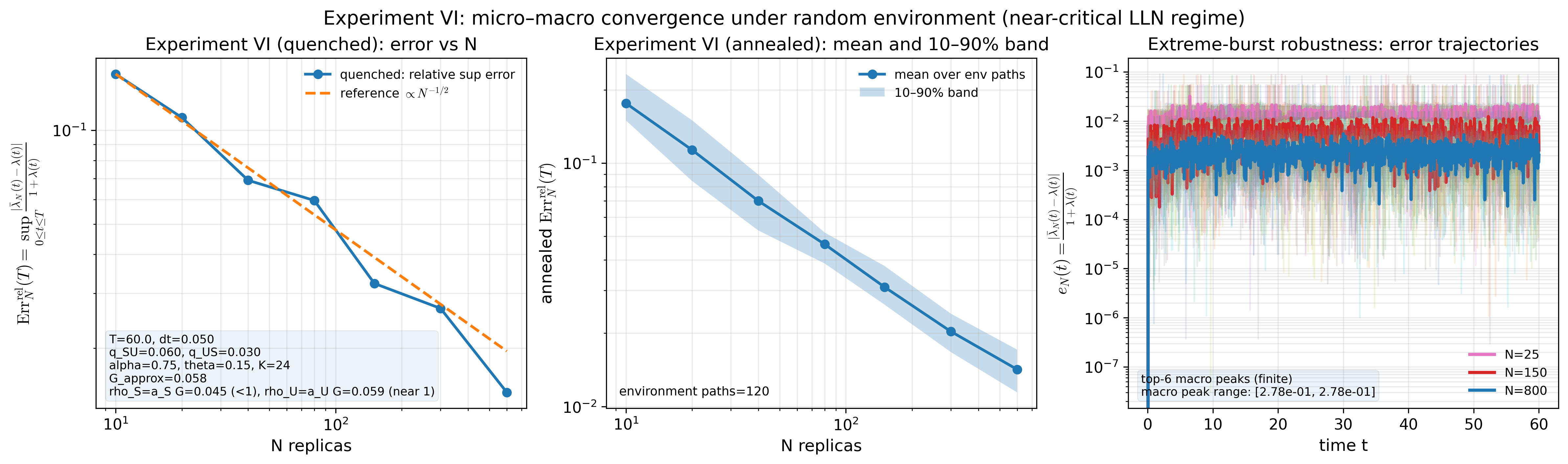}
	\caption{
		\textbf{Experiment VI: numerical micro--macro validation under switching.}
		The relative discrepancy $\mathrm{Err}^{\mathrm{rel}}_N(T)$ decreases systematically with $N$,
		consistent with law-of-large-numbers scaling, even along environment paths that generate
		strong near-critical bursts.
	}
	\label{fig:exp6}
\end{figure}

Experiment~VI confirms that burst amplification and intermittency are \emph{quenched phenomena} driven by the environment and memory,
rather than finite-population artifacts.
The Hawkes population reproduces the same burst structure as the Volterra limit, and the law-of-large-numbers convergence holds uniformly
on compact time intervals even in near-critical regimes.

Across Experiments~I--VI we observe a consistent picture:
annealed summaries can remain controlled on operational horizons while quenched diagnostics show bursty, heavy-tailed amplification driven
by long residence times and long memory.
Moreover, the same burst structure is inherited by the Hawkes population and persists in the macroscopic Volterra limit, supporting the
micro--macro correspondence in random switching environments.

\section{Discussion and outlook}
\label{sec:discussion}

This work develops a stability theory for network-coupled fractional Volterra
dynamics driven by an exogenous Markov switching environment, and identifies a
structural separation between \emph{annealed} (moment/expectation-based) stability
and \emph{quenched} (pathwise) amplification.
The central message is that long-range memory makes residence times dynamically
persistent: atypically long visits to supercritical regimes leave a lasting
imprint through the Volterra convolution, so that stability in expectation may
coexist with strong sample-path growth, intermittency, and heavy-tailed bursts.

Across the analytical and numerical results, intermittency can be read as a
three-stage mechanism:
(i) the environment produces rare but arbitrarily long supercritical exposures
through the residence-time tail of the Markov chain;
(ii) fractional/tempered memory converts exposure length into accumulated
effective gain, because the Volterra kernel retains past forcing and re-injects it
forward in time; and
(iii) the network/operator geometry routes this gain into a small set of spectral
channels, yielding bursts that are heavy-tailed in amplitude yet structured in
direction.
Annealed criteria quantify the \emph{average} gain produced by (i)--(ii), while
quenched results quantify the \emph{pathwise} impact of rare exposures in (i)
amplified by (ii), with (iii) shaping where the amplification concentrates.

\subsection*{Annealed stability: averaged branching and memory-aware Lyapunov margins}

On the annealed side, the analysis shows that subcriticality of each regime is
neither necessary nor sufficient for stability of the switching Volterra system.
Instead, control is governed by an \emph{averaged gain} determined by the
stationary mixture of regime-dependent fractional branching ratios, leading to a
sharp operational criterion in terms of an effective branching parameter
(e.g.\ $\bar\rho<1$).
This extends the classical Hawkes branching condition (based on kernel mass and a
reproduction matrix norm) to an operator-valued Volterra setting on networks with
random coefficients, where the loss of semigroup structure is compensated by
resolvent-family estimates.

A second, complementary annealed mechanism is provided by the memory-adapted
Lyapunov functional: completely monotone kernels admit a tractable energy balance
through their Bernstein representation, yielding mean-square dissipation under an
averaged spectral margin.
From a modelling perspective, this justifies the use of averaged branching and
averaged spectral margins as operational proxies for stability \emph{in mean} and
\emph{in mean square} in long-memory systems subject to regime fluctuations.
From a mathematical perspective, it clarifies which aspects of classical
Lyapunov-switching theory extend to non-Markovian dynamics and which do not.

\subsection*{Quenched amplification: rare residence times, heavy-tailed bursts, and intermittency}

In contrast, the quenched analysis isolates a qualitatively different mechanism.
The interaction between long memory and random sojourn times converts rare but
persistent supercritical excursions into macroscopic bursts.
The burst amplification theorem formalizes the transfer principle:
residence-time statistics of the unstable regime directly shape the tail of burst
amplitudes, producing heavy-tailed extreme-event behaviour even when annealed
moments remain bounded on finite horizons.
In particular, the burst-size distribution inherits polynomial-type decay from
the exit rate of the unstable regime, so that rare long visits become dynamically
dominant along typical paths.

At the level of long-time growth, a subadditive ergodic argument yields the
existence of a deterministic almost sure exponent under
stationarity/ergodicity assumptions, even though explicit formulas are generally
inaccessible in the presence of memory.
This provides a rigorous framework for \emph{intermittency}:
annealed boundedness can mask substantial pathwise risk, and the correct
stability narrative must distinguish expectation control from sample-path growth.
Practically, this indicates that stability assessments based solely on expected
responses (or averaged spectral radii) may systematically underestimate
extreme-event probability in switching environments when memory is strong.

\subsection*{Geometry of bursts on networks: spectral routing, localization, and noncommutativity}

The numerical experiments show that burst amplification is not purely temporal:
it is geometrically organized by the network.
Bursts concentrate along specific spectral channels (Laplacian bands) and can
exhibit strong localization in node space, depending on topology, degree
heterogeneity, and the commutativity structure of excitation operators.

Experiments~IV--V demonstrate two robust phenomena.
First, breaking commutativity induces mode mixing: bursts no longer align with a
single invariant eigenspace, and energy spreads across a restricted band of
modes repeatedly activated across realizations.
Second, despite such mixing, intermittent amplification persists and remains
effectively low-dimensional: burst directions concentrate in a small subspace and
retain heavy-tailed statistics.
These findings connect the present results to themes in networked dynamical
systems---non-normal amplification, spectral localization, and geometry-driven
vulnerability---while emphasizing that memory and switching introduce an
additional layer of organization through history-dependent amplification.

\subsection*{Micro--macro correspondence: branching interpretation and robustness near criticality}

Section~\ref{sec:micro-macro} places the stability and intermittency
mechanisms in a probabilistic context by showing that regime-modulated Hawkes
processes converge, in the large-population limit, to a random-coefficient
Volterra equation exhibiting the same annealed--quenched dichotomy.
This correspondence provides a principled bridge between branching-process
intuition and operator-theoretic Volterra stability analysis: heavy-tailed burst
statistics generated by switching at the microscopic level persist in the
macroscopic limit, rather than being finite-population artifacts.

A delicate issue concerns robustness of this limit in near-critical regimes where
memory and rare unstable sojourns generate pronounced (but finite-horizon) burst
amplification.
Experiment~VI indicates that this does not occur in the near-critical LLN regime:
even along environment paths producing the strongest macroscopic bursts, the
empirical Hawkes mean intensity tracks the Volterra limit with uniformly
controlled relative error and rates consistent with $N^{-1/2}$ scaling.
This supports the Volterra limit as a faithful macroscopic description of large
interacting systems in switching environments, including regimes where rare events
and memory dominate the dynamics.

\subsection*{Relation to switching systems without memory and memory systems without switching}

The annealed--quenched separation identified here is specific to the interaction
of \emph{switching} and \emph{memory}.
In Markovian switching systems without memory, stability can often be inferred
regime-by-regime or via convex combinations of generators or averaged Lyapunov
exponents, and rare residence times do not leave a persistent imprint once the
regime changes.
Conversely, in deterministic or random Volterra systems without switching,
stability is governed by fixed spectral properties of the kernel and feedback
operators, and there is no stochastic amplification mechanism driven by
residence-time fluctuations.
The intermittency mechanism uncovered in this work requires both ingredients:
switching creates a distribution of supercritical exposures, and fractional memory
accumulates and propagates their effect forward in time.
Neither component alone reproduces the observed coexistence of annealed control
and quenched amplification.

\subsection*{Limitations and outlook}

Several directions can sharpen and broaden the present framework.

\paragraph{Sharper intermittency thresholds.}
The current intermittency window is identified through averaged spectral/branching
criteria together with explicit residence-time amplification mechanisms.
Sharper thresholds could be obtained by combining these bounds with large-deviation
estimates for Markov occupation measures and sojourn-time extremes, yielding more
precise phase boundaries for quenched instability and burst-tail exponents.

\paragraph{Beyond exogenous Markov switching.}
State-dependent or endogenous switching (feedback from $U(t)$ into the environment)
would broaden applicability but requires new tools: the environment would no longer
be independent, and quenched arguments must incorporate coupled pathwise dynamics.
Similarly, non-Markov switching or heavy-tailed holding times may strengthen burst
phenomena and connect the theory to renewal and semi-Markov environments.

\paragraph{Nonlinear and data-driven extensions.}
Nonlinear excitation mechanisms, saturation effects, and heterogeneous kernels
raise both modelling and analytical challenges, especially in the presence of
long memory.
In parallel, data-driven calibration---estimating effective branching ratios,
kernel memory mass, and switching rates from observations---would enable the
annealed/quenched decomposition to be used directly as an inference and control
tool in applications.

\paragraph{A concrete theorem-level outlook.}
Two natural theorem targets suggested by the present results are:
(i) a variational characterization of the quenched exponent, combining a
large-deviation principle for occupation/sojourn statistics of $Z(\cdot)$ with a
mode-dependent effective growth functional induced by the kernel resolvent; and
(ii) matching upper/lower tail asymptotics for burst amplitudes $B$ under minimal
regularity assumptions on the completely monotone kernel and on the unstable
regime gain, yielding sharp power-law exponents and pre-factors.
Both directions would refine the intermittency window from qualitative to
quantitative, while preserving the central annealed--quenched separation.

Taken together, these results suggest that annealed stability metrics should be
complemented by quenched, pathwise risk diagnostics in long-memory switching
systems, and that the Volterra-with-environment perspective provides a natural
language for doing so across both microscopic branching models and macroscopic
network dynamics.

	\section{Conclusions}
	\label{sec:conclusions}
	
	We have developed a stability theory for network-coupled fractional Volterra
	systems driven by Markovian regime switching, revealing how long-range memory
	fundamentally alters classical intuition from switching systems without memory.
	The results establish a sharp conceptual dichotomy between annealed stability,
	governed by averaged fractional branching gains and memory-adapted Lyapunov
	functionals, and quenched behaviour, governed by sample-path amplification due to
	rare long residence times in supercritical regimes.
	
	The analysis shows that averaged subcriticality can guarantee boundedness in
	expectation, yet does not preclude pathwise growth. In particular, slow exits from
	unstable regimes generate heavy-tailed burst statistics and yield an almost sure
	growth exponent characterised via a subadditive ergodic argument, thereby providing
	a rigorous framework for intermittency and metastability in nonlocal random
	dynamical systems.
	
	Finally, we established a micro--macro correspondence in random environments:
	regime-modulated Hawkes processes converge to random-coefficient Volterra dynamics
	in both annealed and quenched senses, transferring burst mechanisms from microscopic
	branching dynamics to macroscopic long-memory flows. Together with the numerical
	experiments, these results indicate that memory-aware, pathwise risk metrics are
	essential for assessing robustness of switching systems, particularly on networks
	where geometry shapes spectral routing and localisation of burst amplification.
	
\section*{Appendix}
	\appendix
	
	\section{Measurability and adaptedness of mild solutions}
	\label{app:adaptedness}
	
	Let $\{\mathcal F_t\}_{t\ge 0}$ be the (completed, right-continuous) natural filtration
	generated by the regime process $Z(\cdot)$, which has c\`adl\`ag sample paths.
	For each fixed realization $\omega$, Theorem~\ref{thm:pathwise-wp} yields a unique
	continuous mild solution $t\mapsto U(t,\omega)\in\mathcal X$.
	
	\begin{proposition}[Progressive measurability and adaptedness]
		\label{prop:adaptedness}
		Assume the hypotheses of Theorem~\ref{thm:pathwise-wp}.
		Then the mapping $(t,\omega)\mapsto U(t,\omega)$ is jointly measurable
		$\big([0,T]\times\Omega,\mathcal B([0,T])\otimes\mathcal F_T\big)\to(\mathcal X,\mathcal B(\mathcal X))$
		for each $T>0$, and $\{U(t)\}_{t\ge 0}$ is $\{\mathcal F_t\}$-adapted.
		In fact, $U$ is progressively measurable with respect to $\{\mathcal F_t\}$.
	\end{proposition}
	
	\begin{proof}[Proof sketch]
		Fix $T>0$.
		The pathwise construction in Theorem~\ref{thm:pathwise-wp} can be realized by a
		Picard iteration on $C([0,T];\mathcal X)$:
		set $U^{(0)}(t,\omega)\equiv U_0$, and define recursively $U^{(m+1)}(\cdot,\omega)$
		as the unique mild solution of the frozen-environment linear Volterra equation
		with coefficients evaluated along $Z(\cdot,\omega)$ and with the memory term
		driven by $U^{(m)}(\cdot,\omega)$.
		
		For each $m$, the map $\omega\mapsto Z(\cdot,\omega)\big|_{[0,T]}$ is measurable
		into the Skorokhod space $D([0,T];\mathcal Z)$, and the mild-solution map for the
		corresponding linear Volterra equation is continuous with respect to the driving
		environment path on compact time intervals under the standing assumptions
		(boundedness of coefficients, complete monotonicity/finite mass of kernels, and
		well-posedness of the resolvent family). Hence $(t,\omega)\mapsto U^{(m)}(t,\omega)$
		is jointly measurable and, since $U^{(m)}(t,\omega)$ depends on $\omega$ only
		through $Z(s,\omega)$ for $s\le t$, it is $\mathcal F_t$-measurable for each $t$.
		
		The contraction estimate in Theorem~\ref{thm:pathwise-wp} implies that
		$U^{(m)}\to U$ uniformly on $[0,T]$ for each $\omega$.
		Pointwise limits of jointly measurable maps are jointly measurable, so
		$(t,\omega)\mapsto U(t,\omega)$ is jointly measurable.
		Moreover, as each $U^{(m)}(t)$ is $\mathcal F_t$-measurable and $U^{(m)}(t)\to U(t)$
		in $\mathcal X$, the limit $U(t)$ is $\mathcal F_t$-measurable.
		Progressive measurability follows from joint measurability on $[0,t]\times\Omega$
		and adaptedness for each $t\le T$.
	\end{proof}
	
	As a consequence of Proposition~\ref{prop:adaptedness}, all expectations,
	conditional expectations, and stopping-time constructions used in
	Sections~\ref{sec:mean-stab}--\ref{sec:as-stab} are well-defined.

	\section{Technical proofs for the memory Lyapunov approach}
	\label{app:lyapunov-details}
	
	\subsection{Derivative identities and closure of the dissipation estimate}
	\label{app:derivative-closure}
	
	This appendix records the detailed Lyapunov calculation underlying
	Theorem~\ref{thm:lyapunov-mean-stability}.
	
	\begin{proof}[Proof of Theorem~\ref{thm:lyapunov-mean-stability}]
		Fix a sample path $z(\cdot)=Z(\cdot,\omega)$ and consider the augmented system
		\eqref{eq:main-switching} with $F\equiv 0$.
		
		\smallskip\noindent
		\emph{Step 1: Derivative identities.}
		Differentiating $\|U(t)\|_{\mathcal X}^2$ gives
		\begin{equation}
			\label{app:eq:dU2}
			\frac12\frac{d}{dt}\|U(t)\|_{\mathcal X}^2
			=
			\langle \mathcal B U(t),U(t)\rangle_{\mathcal X}
			+
			\left\langle
			\int_{(0,\infty)} \mathcal G^\sharp_{z(t)}\,w(t,r)\,\nu_{z(t)}(dr),
			U(t)\right\rangle_{\mathcal X}.
		\end{equation}
		From \eqref{eq:aux-memory},
		\[
		\frac12\frac{d}{dt}\|w(t,r)\|_{\mathcal X}^2
		=
		\langle U(t),w(t,r)\rangle_{\mathcal X}
		-
		r\|w(t,r)\|_{\mathcal X}^2.
		\]
		Multiplying by $\eta r$ and integrating against $\nu_{z(t)}(dr)$ yields
		\begin{equation}
			\label{app:eq:dw2}
			\frac{\eta}{2}\frac{d}{dt}\int r\|w(t,r)\|_{\mathcal X}^2\,\nu_{z(t)}(dr)
			=
			\eta \int r\,\langle U(t),w(t,r)\rangle_{\mathcal X}\,\nu_{z(t)}(dr)
			-
			\eta \int r^2\|w(t,r)\|_{\mathcal X}^2\,\nu_{z(t)}(dr).
		\end{equation}
		
		\smallskip\noindent
		\emph{Step 2: Pathwise energy inequality.}
		Adding \eqref{app:eq:dU2} and \eqref{app:eq:dw2} and using \eqref{eq:strict-dissip} gives
		\begin{align}
			\label{app:eq:Vprime-pathwise}
			\frac12\frac{d}{dt}\mathcal V_{z(t)}(t)
			&\le
			-\beta\|U(t)\|_{\mathcal X}^2
			+
			\left\langle
			\int \mathcal G^\sharp_{z(t)}\,w(t,r)\,\nu_{z(t)}(dr),
			U(t)\right\rangle_{\mathcal X}\nonumber\\
			&\quad
			+\eta \int r\,\langle U(t),w(t,r)\rangle_{\mathcal X}\,\nu_{z(t)}(dr)
			-
			\eta \int r^2\|w(t,r)\|_{\mathcal X}^2\,\nu_{z(t)}(dr).
		\end{align}
		
		\smallskip\noindent
		\emph{Step 3: Bounds on coupling terms.}
		Since $\mathcal G^\sharp_z$ acts through $A_z$ on the memory channel, there exists
		$C_G>0$ (depending only on the block structure of $\mathcal X$) such that
		\[
		\|\mathcal G^\sharp_z v\|_{\mathcal X}\le C_G\|A_z\|\,\|v\|_{\mathcal X},
		\qquad \forall v\in\mathcal X.
		\]
		Hence,
		\[
		\left|\left\langle
		\int \mathcal G^\sharp_{z(t)}\,w(t,r)\,\nu_{z(t)}(dr),
		U(t)\right\rangle_{\mathcal X}\right|
		\le
		C_G\|A_{z(t)}\|\,\|U(t)\|_{\mathcal X}
		\int \|w(t,r)\|_{\mathcal X}\,\nu_{z(t)}(dr).
		\]
		By Cauchy--Schwarz with weights $r$ and $r^{-1}$,
		\[
		\int \|w(t,r)\|_{\mathcal X}\,\nu_{z(t)}(dr)
		\le
		\left(\int r\,\|w(t,r)\|_{\mathcal X}^2\,\nu_{z(t)}(dr)\right)^{1/2}
		\left(\int \frac{1}{r}\,\nu_{z(t)}(dr)\right)^{1/2}.
		\]
		Using $\int r^{-1}\nu_{z(t)}(dr)=G_{z(t)}$ yields
		\[
		\left|\left\langle
		\int \mathcal G^\sharp_{z(t)}\,w(t,r)\,\nu_{z(t)}(dr),
		U(t)\right\rangle_{\mathcal X}\right|
		\le
		C_G\,\|A_{z(t)}\|\sqrt{G_{z(t)}}\,
		\|U(t)\|_{\mathcal X}
		\left(\int r\,\|w(t,r)\|_{\mathcal X}^2\,\nu_{z(t)}(dr)\right)^{1/2}.
		\]
		By Young's inequality, for any $\varepsilon>0$,
		\begin{equation}
			\label{app:eq:young}
			\left|\left\langle
			\int \mathcal G^\sharp_{z(t)}\,w(t,r)\,\nu_{z(t)}(dr),
			U(t)\right\rangle_{\mathcal X}\right|
			\le
			\varepsilon \|U(t)\|_{\mathcal X}^2
			+
			\frac{C_G^2}{4\varepsilon}\,\|A_{z(t)}\|^2\,G_{z(t)}
			\int r\,\|w(t,r)\|_{\mathcal X}^2\,\nu_{z(t)}(dr).
		\end{equation}
		Similarly,
		\[
		\eta \int r\,\langle U(t),w(t,r)\rangle_{\mathcal X}\,\nu_{z(t)}(dr)
		\le
		\eta\varepsilon \|U(t)\|_{\mathcal X}^2
		+
		\frac{\eta}{4\varepsilon}\int r\,\|w(t,r)\|_{\mathcal X}^2\,\nu_{z(t)}(dr).
		\]
		
		\smallskip\noindent
		\emph{Step 4: Closing the estimate and taking expectations.}
		Substituting into \eqref{app:eq:Vprime-pathwise} yields
		\begin{align}
			\label{app:eq:Vprime-pre}
			\frac12\frac{d}{dt}\mathcal V_{z(t)}(t)
			&\le
			-\big(\beta-(1+\eta)\varepsilon\big)\|U(t)\|_{\mathcal X}^2\nonumber\\
			&\quad
			+\Big(
			\frac{C_G^2}{4\varepsilon}\|A_{z(t)}\|^2G_{z(t)}
			+\frac{\eta}{4\varepsilon}
			\Big)\int r\,\|w(t,r)\|_{\mathcal X}^2\,\nu_{z(t)}(dr)
			-\eta\int r^2\|w(t,r)\|_{\mathcal X}^2\,\nu_{z(t)}(dr).
		\end{align}
		Choose $\varepsilon>0$ so that $\beta-(1+\eta)\varepsilon>0$.
		Moreover, by splitting $(0,\infty)=(0,1)\cup[1,\infty)$, the negative term
		$-\eta\int r^2\|w\|^2$ controls $-\eta\int_{[1,\infty)} r\|w\|^2$, while on $(0,1)$
		the factor $r$ is small and can be absorbed into the $\|U\|^2$ dissipation by
		choosing $\eta$ sufficiently large (depending only on the uniform bounds in
		Assumptions~\ref{ass:kernels}--\ref{ass:uniform}).
		Thus one obtains a pathwise inequality of the form
		\begin{equation}
			\label{app:eq:Vprime-closed}
			\frac{d}{dt}\mathcal V_{z(t)}(t)
			\le
			-\,c_0\|U(t)\|_{\mathcal X}^2
			+
			C_0\,\rho_{z(t)}\,\mathcal V_{z(t)}(t),
		\end{equation}
		for constants $c_0,C_0>0$ independent of $t$ and $z$.
		
		Taking expectations in \eqref{app:eq:Vprime-closed} and using stationarity of
		$Z(t)$ yields
		\[
		\frac{d}{dt}\mathbb E[\mathcal V_{Z(t)}(t)]
		\le
		-\,c_0\,\mathbb E\|U(t)\|_{\mathcal X}^2
		+
		C_0\,\mathbb E[\rho_{Z(t)}\,\mathcal V_{Z(t)}(t)].
		\]
		Using $\rho_{Z(t)}\le \rho_{\max}$ and the averaged subcriticality condition of
		Theorem~\ref{thm:lyapunov-mean-stability}, one can choose the parameters so that
		the positive term is dominated in expectation, leading to the dissipation bound
		\eqref{eq:lyapunov-dissipation}; integrating on $[0,T]$ yields
		\eqref{eq:mean-square-bound}.
	\end{proof}
	
	\subsection{Generator identity for memory-dependent Lyapunov functionals}
	\label{app:generator}
	
	This appendix records a generator identity used to justify Lyapunov arguments
	for regime-switching dynamics with memory.
	
	\begin{lemma}[Infinitesimal generator for memory-dependent Lyapunov functionals]
		\label{app:lem:generator}
		Let $V:\mathcal{X}\times\mathcal{Z}\to\mathbb{R}_+$ be the Lyapunov functional
		defined by
		\begin{equation}
			\label{app:eq:Lyap-memory}
			V(u,z)
			=
			\|u\|_{\mathcal{X}}^2
			+
			\int_0^\infty \phi_z(s)\,\|u(t-s)\|_{\mathcal{X}}^2\,ds,
		\end{equation}
		where $\phi_z:\mathbb{R}_+\to\mathbb{R}_+$ is locally integrable, completely
		monotone, and $\int_0^\infty \phi_z(s)\,ds<\infty$.
		Then $V$ belongs to the domain of the infinitesimal generator $\mathcal{L}$ of
		the Markov process $(u(t),Z(t))$, and its action decomposes as
		\begin{equation}
			\label{app:eq:generator-decomposition}
			\mathcal{L}V(u,z)
			=
			\mathcal{L}_{\mathrm{det}}V(u,z)
			+
			\mathcal{L}_{\mathrm{jump}}V(u,z),
		\end{equation}
		where
		\begin{align}
			\label{app:eq:generator-det}
			\mathcal{L}_{\mathrm{det}}V(u,z)
			&=
			2\langle u,\mathcal{A}_z u\rangle_{\mathcal{X}}
			+
			2\Big\langle
			u,
			\int_0^\infty K_z(s)\,\mathcal{B}_z u(t-s)\,ds
			\Big\rangle_{\mathcal{X}}
			\nonumber\\
			&\quad
			-
			\int_0^\infty \phi_z'(s)\,\|u(t-s)\|_{\mathcal{X}}^2\,ds,
		\end{align}
		and
		\begin{equation}
			\label{app:eq:generator-jump}
			\mathcal{L}_{\mathrm{jump}}V(u,z)
			=
			\sum_{z'\neq z} q_{zz'}\big(V(u,z')-V(u,z)\big).
		\end{equation}
	\end{lemma}
	
	\begin{proof}
		We decompose the argument into deterministic and jump contributions.
		
		\medskip
		\noindent\emph{Step 1: Deterministic Volterra contribution.}
		Fix $z\in\mathcal{Z}$ and consider the evolution under the frozen regime
		$Z(t)\equiv z$.
		By resolvent-family theory for Volterra equations with completely monotone kernels,
		mild solutions $u(t)$ are continuous in $\mathcal{X}$ and locally absolutely
		continuous in time.
		Differentiating $\|u(t)\|_{\mathcal{X}}^2$ along trajectories yields
		\[
		\frac{d}{dt}\|u(t)\|_{\mathcal{X}}^2
		=
		2\langle u(t),\mathcal{A}_z u(t)\rangle_{\mathcal{X}}
		+
		2\Big\langle
		u(t),
		\int_0^t K_z(t-s)\,\mathcal{B}_z u(s)\,ds
		\Big\rangle_{\mathcal{X}}.
		\]
		For the memory term, differentiation under the integral sign and complete
		monotonicity of $\phi_z$ yield $\phi_z'\le 0$ in the sense of distributions and
		\[
		\frac{d}{dt}
		\int_0^\infty \phi_z(s)\,\|u(t-s)\|_{\mathcal{X}}^2\,ds
		=
		-
		\int_0^\infty \phi_z'(s)\,\|u(t-s)\|_{\mathcal{X}}^2\,ds.
		\]
		Combining both contributions gives \eqref{app:eq:generator-det}.
		
		\medskip
		\noindent\emph{Step 2: Jump contribution.}
		Since $V$ depends on the regime only through $z$, the jump part of the generator is
		\[
		\mathcal{L}_{\mathrm{jump}}V(u,z)
		=
		\sum_{z'\neq z} q_{zz'}\big(V(u,z')-V(u,z)\big),
		\]
		which is \eqref{app:eq:generator-jump}.
		
		\medskip
		\noindent\emph{Step 3: Domain considerations.}
		The assumed integrability of $\phi_z$ and the boundedness of the resolvent family
		ensure all terms are finite and measurable, hence $V\in\mathrm{Dom}(\mathcal L)$ and
		\eqref{app:eq:generator-decomposition} holds.
	\end{proof}

	\section{Auxiliary quenched proofs: burst tails, intermittency, and almost sure growth}
	\label{app:quenched-proofs}
	
	\subsection*{Proof of Theorem~\ref{thm:burst-qUS}}
	\label{app:proof-burst}
	\begin{proof}[Proof] 
		Since $\tau_U\sim\mathrm{Exp}(q_{US})$, we have $\mathbb P(\tau_U>t)=e^{-q_{US}t}$ for $t\ge 0$.
		For the moment criterion, note that $B^p=e^{p\gamma_U\tau_U}$ and hence
		\[
		\mathbb E[B^p]=\mathbb E[e^{p\gamma_U\tau_U}]
		=
		\int_0^\infty e^{p\gamma_U t}\,q_{US}e^{-q_{US}t}\,dt
		=
		\frac{q_{US}}{q_{US}-p\gamma_U},
		\]
		which is finite iff $p\gamma_U<q_{US}$ and diverges otherwise.
		For the tail, for $b\ge 1$,
		\[
		\mathbb P(B>b)=\mathbb P\!\left(e^{\gamma_U\tau_U}>b\right)
		=\mathbb P\!\left(\tau_U>\frac{\log b}{\gamma_U}\right)
		=\exp\!\left(-q_{US}\frac{\log b}{\gamma_U}\right)
		=b^{-q_{US}/\gamma_U}.
		\]
	\end{proof}
	
	\subsection*{Proof sketch for Proposition~\ref{prop:intermittency-window-2reg}}
	\label{app:proof-intermittency-window}
	\begin{proof}[Proof sketch]
		Let $\{\tau_U^{(k)}\}_{k\ge1}$ be the i.i.d.\ sojourn times in $U$ across successive visits,
		with $\tau_U^{(k)}\sim\mathrm{Exp}(q_{US})$.
		Fix a threshold sequence $b_k\uparrow\infty$ and consider the events
		\[
		E_k:=\{B_k>b_k\},\qquad B_k:=\exp(\gamma_U\tau_U^{(k)}).
		\]
		By Theorem~\ref{thm:burst-qUS}, $\mathbb P(E_k)=b_k^{-q_{US}/\gamma_U}$.
		Choosing $b_k=k^{\gamma_U/q_{US}}$ gives $\mathbb P(E_k)=k^{-1}$ and hence
		$\sum_k \mathbb P(E_k)=\infty$.
		Since the $\tau_U^{(k)}$ are independent, Borel--Cantelli implies
		$\mathbb P(E_k\ \text{i.o.})=1$.
		On each occurrence of $E_k$, the unstable visit produces a multiplicative amplification at least $b_k$,
		so the running supremum of the norm diverges along those paths.
		Combining this quenched divergence mechanism with annealed boundedness under the
		averaged subcriticality condition in Section~\ref{sec:mean-stab} yields the
		annealed--quenched intermittency claim of Proposition~\ref{prop:intermittency-window-2reg}
		(and Definition~\ref{def:intermittency}, where used).
	\end{proof}
	
	\subsection*{Proof of Theorem~\ref{thm:subadditive}}
	\label{app:proof-subadditive}
	\begin{proof}[Proof]
		Let $(\Omega,\mathcal F,\mathbb P,(\theta_t)_{t\ge0})$ be the stationary metric dynamical system
		generated by the regime process, with $\theta_t$ denoting time shift.
		Let $\mathcal U(t,\omega)$ denote the linear cocycle (solution operator) associated with the
		pathwise mild dynamics.
		Define, for $0\le s\le t$,
		\[
		a_{s,t}(\omega):=\log\|\mathcal U(t-s,\theta_s\omega)\|.
		\]
		The cocycle property $\mathcal U(t,\omega)=\mathcal U(t-s,\theta_s\omega)\,\mathcal U(s,\omega)$ implies
		subadditivity:
		\[
		a_{0,t}(\omega)\le a_{0,s}(\omega)+a_{s,t}(\omega).
		\]
		Under $\mathbb E[\log^+\|\mathcal U(1,\cdot)\|]<\infty$, Kingman's subadditive ergodic theorem yields the
		existence of a deterministic constant $\gamma$ such that
		\[
		\lim_{t\to\infty}\frac{1}{t}\,a_{0,t}(\omega)=\gamma
		\qquad \text{for $\mathbb P$-almost every }\omega.
		\]
		Finally, since $U(t,\omega)=\mathcal U(t,\omega)U_0$, we have
		$\|U(t,\omega)\|_{\mathcal X}\le \|\mathcal U(t,\omega)\|\,\|U_0\|_{\mathcal X}$, implying the asserted
		almost sure growth-rate bound and completing the proof.
	\end{proof}

	\section{Technical proofs for the Hawkes micro--macro limits}
	\label{app:hawkes}
	
	\subsection{Martingale estimates for the annealed limit}
	Let $\bar M_N(t)=N^{-1}\sum_{k=1}^N M_N^{(\cdot),k}(t)$ denote the averaged martingale term
	arising in the $N$-replica Hawkes representation.
	Using Doob's maximal inequality and orthogonality across replicas, one obtains for each $T>0$
	\[
	\mathbb E\!\left[
	\sup_{0\le t\le T}
	\Big\|
	\int_0^t g_{Z(t)}(t-s)\,d\bar M_N(s)
	\Big\|
	\right]
	\le
	C_T\,N^{-1/2},
	\]
	where $C_T<\infty$ depends only on $T$ and the uniform bounds on $A_z$ and $g_z$.
	This yields the annealed convergence of the martingale remainder to zero uniformly on $[0,T]$,
	as used in Theorem~\ref{thm:annealed-hawkes}.
	
	\subsection{Quenched strong law for the averaged martingale}
	Fix a realization $z(\cdot)=Z(\cdot,\omega)$.
	Conditional on $z(\cdot)$, the martingales $\{M_N^{(\cdot),k}\}_{k\ge1}$ are i.i.d.\ with
	uniformly bounded quadratic variation on compact intervals.
	A martingale strong law of large numbers yields
	\[
	\sup_{0\le t\le T}
	\Big\|
	\int_0^t g_{z(t)}(t-s)\,d\bar M_N(s)
	\Big\|
	\xrightarrow[N\to\infty]{\mathrm{a.s.}} 0,
	\qquad \forall T>0,
	\]
	which provides the quenched martingale control required in Theorem~\ref{thm:quenched-hawkes}.

	\bibliographystyle{unsrt}
	\bibliography{references}
	
\end{document}